%% file: Rk1-Amice.tex
\documentclass{NP-IV}
\usepackage[latin1]{inputenc}
\usepackage{amsmath}
\usepackage{amscd}
\usepackage{mathrsfs}
\usepackage[all]{xy}
\usepackage{graphicx}
\usepackage{pstricks} 
\usepackage{amsfonts}   
\usepackage{pict2e}   
\usepackage{makeidx}
\usepackage[colorlinks=true,linkcolor=red,citecolor=red]{hyperref}

\usepackage{amssymb}
\usepackage{latexsym}
\usepackage{bold-extra}

\input{raccourcisandrea}
\usepackage{amsthm}
\swapnumbers

\makeatletter
\def\swappedhead#1#2#3{%
  \thmname{#1}\;%
  \thmnumber{\@upn{\the\thm@headfont#2\@ifnotempty{#1}}}%
  \thmnote{\,{\the\thm@notefont(#3)}}{.~}}
\makeatother

\newtheoremstyle{dotless-thm}
  {10pt}
  {10pt}
  {\itshape}
  {}
  {\bfseries}
  {}
  {.0em}
  {}
\theoremstyle{dotless-thm}
\newtheorem{theorem}{\textbf{Theorem}}[subsection]
\newtheorem{def-intro}{\textbf{\textsc{Definition}}}
\newtheorem{thm-intro}{\textbf{\textsc{Theorem}}}
\newtheorem{rk-intro}[thm-intro]{\textbf{\textsc{Remark}}}
\newtheorem{cor-intro}[thm-intro]{\textbf{\textsc{Corollary}}}
\newtheorem{proposition}[theorem]{\textbf{Proposition}}
\newtheorem{lemma}[theorem]{\textbf{Lemma}}
\newtheorem{corollary}[theorem]{\textbf{Corollary}}
\newtheorem{definition}[theorem]{\textbf{Definition}}
\newtheorem{remark}[theorem]{\textbf{Remark}}

\newtheorem{hypothesis}[theorem]{\textbf{Hypothesis}}

\newtheorem{notation}[theorem]{\textbf{Notation}}
\numberwithin{equation}{section}

\title[Solvability of rank one $p$-adic differential and 
$q$-difference equations over the Amice ring]{Solvability 
of rank one $p$-adic differential and $q$-difference 
equations over the Amice ring}

\author{Andrea Pulita}
\email{pulita@math.univ-montp2.fr}
\address{Département de Mathématiques,
Université de Montpellier II, CC051,
Place Eugène Bataillon,
F-34 095, Montpellier CEDEX 5.}

\date{\today}



\begin{abstract}
We provide a necessary and sufficient condition for the solvability of a rank one differential (resp. $q$-difference) equation over the Amice's ring. We also extend to that ring a Birkoff decomposition result, originally due to Motzkin.
\end{abstract}

\begin{document}
\maketitle

\begin{center}
Version of \today
\end{center}

\makeatletter
\renewcommand\tableofcontents{%
    \subsection*{\contentsname}%
    \@starttoc{toc}%
    }
\makeatother

\begin{small}
\setcounter{tocdepth}{3} \tableofcontents
\end{small}

\setcounter{section}{0}


\section*{\textsc{Introduction}}
\addcontentsline{toc}{section}{\textsc{Introduction}}
Let $(K,|.|)$ be a  
field of characteristic $0$ 
which is complete with respect to an 
ultramentric absolute value $|.|$, and whose 
residual field $k$ has positive characteristic $p>0$. 
Denote by $\O_K:=\{x\in K\;|\;|x|\leq 1\}$ its ring of 
integers.

The Robba ring $\mathfrak{R}_K$ is the ring of power 
series $f(T)=\sum_{i\in\mathbb{Z}}a_iT^i$, $a_i\in K$, 
for which there exists an unspecified $\varepsilon<1$ 
(depending on $f$) 
such that $f(T)$ converges on the annulus 
$\{\varepsilon<|T|<1\}$. 
In a previous work \cite{Rk1} (see also \cite{Rk1-NP}) 
we described the isomorphism classes of rank one 
solvable differential equations over $\mathfrak{R}_K$. 
In particular we have obtained a criterion permitting to 
read in the coefficients of the differential equation the 
solvability.

\if{
under the assumption that $K$ contains the $p^m$-roots 
of unity $\bs{\mu}_{p^m}$, for all $m$ large enough. 
Namely, let $\CW$ denote the additive group of Witt 
co-vectors, and let $\Fb$ be its Frobenius, then the 
group $\mathrm{Pic}^{\mathrm{sol}}(
\mathfrak{R}_K)$, under tensor product, of the 
isomorphism classes of rank one solvable differential 
equations is given by 
\begin{equation}\label{eq : Pic sol}
\mathrm{Pic}^{\mathrm{sol}}(
\mathfrak{R}_K)\;:=\;
\frac{\CW(T^{-1}k[T^{-1}])}{
(\Fb-1)\CW(T^{-1}k[T^{-1}])}
\oplus\frac{\mathbb{Z}_p}{\mathbb{Z}}\;.
\end{equation}
}\fi
In another work \cite{Pu-q-Diff} 
(see also \cite{Diff-Gamma}) we studied 
the phenomena of deformation of $q$-difference 
equations and we have proved that, under the solvability 
condition, the category of differential equation is 
equivalent to that of $q$-difference equations (this 
generalizes previous works of Yves André 
and Lucia Di Vizio 
\cite{An-DV}, \cite{DV-Dwork}).

In this paper we are interested to differential and 
$q$-difference equations over the Amice's ring 
$\mathcal{E}_K$. This ring is 
formed by formal power series 
$f(T)=\sum_{i\in\mathbb{Z}}a_iT^i$, $a_i\in K$, that 
are bounded (i.e. $\sup_i|a_i|<+\infty$), and such that 
$\lim_{i\to-\infty}|a_i|=0$. It is the ring used by 
J.M.Fontaine in the theory of $(\phi,\Gamma)$-modules 
\cite{Fo}. 

A classification  of rank one 
differential (or $q$-difference) equations over the ring 
$\mathcal{E}_K$ is not known, and it seems reasonable 
to think that such a classification 
will be quite different in nature with respect to that 
obtained in \cite{Rk1} for differential 
equations over the Robba ring $\mathfrak{R}_K$. 
This will not be the goal of 
this paper. We here obtain a criterion of solvability 
for differential and $q$-difference equations 
similar to that in \cite{Rk1}.  

We actually describe completely the precise nature of the 
solutions of differential and difference equations as 
exponentials of Artin-Hasse type.

As a corollary we obtain that every differential equation 
over $\mathcal{E}_K$ has a basis in which the 
associated operator has coefficients in 
$\O_K[[T^{-1}]]$. 
This constitutes an analogous of the Katz canonical 
extension theorem \cite{Katz-Can} (see also \cite{Matsuda-Unipotent}).\\

The results of this paper have been obtained in 2005, 
during our PhD at the university of Paris, under the 
supervision of Gilles  Christol. \\

\textbf{Acknowledgments}
Step 4 in the proof of Proposition 
\ref{division of the problem over Amice} is due to 
Gilles Christol, 
we want here to express our gratitude to him for helpful 
discussions.

\section*{\textsc{First part : 
solvability of rank one differential equations over 
$\mathcal{E}_K$}}
\addcontentsline{toc}{section}{\textsc{First part : 
solvability of rank one differential equations over 
$\mathcal{E}_K$}}

\section{Notations}
Let $\mathbb{R}_{\geq 0}$ be the interval of real 
numbers that are greater than or equal to $0$. 

Let $K$ be a complete valued field of characteristic $0$, 
with ring of integers $\O_K:=\{x\in K,|x|\leq 1\}$, and 
maximal ideal $\mathfrak{p}_K:=\{x\in K,|x|< 1\}$. We 
assume that the residual field 
$k:=\O_K/\mathfrak{p}_K$ has positive 
characteristic $p>0$.

If $I\subseteq\mathbb{R}_{\geq 0}$ is any interval, we 
denote by $\a_{K}(I)$ the ring of analytic functions on 
the space $\{|T|\in I\}$. If $0\in I$ this is an open or 
closed disk, in this case we have
\begin{equation}
\a_K(I)\;:=\;\{\sum_{i\geq 0}a_iT^i\;,\;a_i\in K,
\lim_{i\to\infty}|a_i|\rho^i=0,\textrm{ for all }
\rho\in I\}\;.
\end{equation}
If $0\notin I$ it is an open, closed, or semi-open 
annulus and we have
\begin{equation}
\a_K(I)\;:=\;\{\sum_{i\in\mathbb{Z}}a_iT^i\;,\;
a_i\in K,
\lim_{i\to\pm\infty}|a_i|\rho^i=0,
\textrm{ for all }\rho\in I\}\;.
\end{equation}
For all $\rho\in I$ we have a norm on $\a_K(I)$ given by 
$|\sum_{i\in\mathbb{Z}} a_i T^i|_\rho:=\sup_i|a_i|\rho^i$. 
And $\a_K(I)$ is complete with respect to the 
Frechet 
topology defined by the family of norms 
$\{|.|_\rho\}_{\rho\in I}$. We define the 
\emph{Robba ring} as 
\begin{equation}
\mathfrak{R}_K\;:=\;
\cup_{\varepsilon>0}\a_K(]1-\varepsilon,1[)\;.
\end{equation}
The topology of the ring $\mathfrak{R}_K$ is the limit of 
the topologies of $\a_K(]1-\varepsilon,1[)$ which are 
Frechet spaces. It is hence a $\mathcal{LF}$ topology.

The \emph{Amice's ring} $\mathcal{E}_K$ is defined as
\begin{equation}
\mathcal{E}_K\;:=\;\{\sum_{i\in\mathbb{Z}}a_iT^i,\; 
a_i\in K,\; \sup_i|a_i|<+\infty,\; \lim_{i\to-\infty}|a_i|
=0\}\;.
\end{equation}
It is a complete valued ring with respect to 
the Gauss norm $|\sum a_iT^i|_1:=\sup|a_i|$. 
Its ring of integers $\O_{\mathcal{E}_K}=
\{f\in\mathcal{E}_K\;|\;|f|_1\leq 1\}$ is a local 
ring, with residual field $k((t))$ 
(i.e. a field of Laurent power series with coefficients in 
$k$). 
If $K$ is discretely valued, $\mathcal{E}_K$ is 
moreover a field.

We define the \emph{bounded Robba ring} as 
$\Ed_K:=\mathfrak{R}_K\cap\mathcal{E}_K$. If $K$ is 
discretely valued, it is a field. $\mathfrak{R}_K$ and 
$\mathcal{E}_K$ induce two distinct topologies on 
$\Ed_K$, and this last  is dense in $\mathfrak{R}_K$ 
and in $\mathcal{E}_K$ with respect to the 
corresponding topologies.

\subsection{Differential modules and radius of 
convergence}

Let $A$ be one of the rings $\a_K(I)$ or 
$\mathcal{E}_K$. The $A$-module of continuous 
differentials $\Omega^1_{A/K}$ is 
free and one dimensional over $A$. 
Let $d:A\to A$ be a non trivial derivation corresponding 
to a generator of $\Omega^1_{A/K}$. 
A differential module over $A$ is a finite free 
$A$-module $M$, together with a linear map
$\nabla:M\to M$, called \emph{connection}, satisfying 
the Leibniz rule $\nabla(fm)=d(f)m+f\nabla(m)$, 
$f\in A$, $m\in M$.

In this paper we will always assume the rank of $M$ to 
be $1$. We denote by $\d:= T\frac{d}{dT}$. 
If a basis of $M$ is given, 
then $\nabla$ becomes an operator of the form 
$f\mapsto \d(f)-g\cdot f:A\to A$, where $g\in A$. 
We say then that $M$ is defined by the operator $\d-g$. 
With respect to another basis $M$ will be represented by 
another operator $\d-g_2$, and $g_2$ is related to $g$ 
by the rule $g_2=g+\frac{\d(h)}{h}$, where 
$h\in A^{\times}$ is the 
base change matrix.

We denote by $M_1\otimes M_2$ the tensor product of 
two differential modules 
$(M_1,\nabla_1)$ and $(M_2,\nabla_2)$. This is a 
differential module whose underling $A$-module is 
$M_1\otimes_A M_2$, and whose connection 
is $\nabla_1\otimes\mathrm{Id}+
\mathrm{Id}\otimes\nabla_2$. 
If $\d-g_1$ and $\d-g_2$ are associated operators with 
respect to some bases, then $\d-(g_1+g_2)$ will be the 
operator of $M_1\otimes M_2$ with respect to the tensor 
product of the bases.

Let now $\d-g(T)$ be a differential operator with 
$g\in\a_K(I)$, 
and let $\Omega/K$ be any complete valued 
field extension of $K$. For all $x\in\Omega$, $|x|\in I$, 
we look at $\Omega[[T-x]]$ as an $\a_K(I)-$differential 
algebra by the Taylor map 
\begin{equation}\label{eq : Taylor solution}
f(T)\mapsto\sum_{k\geq
0}(\frac{d}{dT})^k(f)(x)\frac{(T-x)^k}{k!}\;:\;
\a_K(I)\longrightarrow\Omega[[T-x]]\;.
\end{equation}
Define inductively $g_{[k]}(T)$ as $g_{[0]}:=1$, 
$g_{[1]}:=g(T)/T$, and for all $k\geq 1$ we set 
$g_{[k+1]}:=\frac{d}{dT}(g_{[n]})+g_{[k]}g_{[1]}$.
The Taylor solution of $\d-g(T)$ at $x$ is then
\begin{equation}\label{s_x(T)}
s_x(T):=\sum_{k\geq 0} g_{[k]}(x)\frac{(T-x)^k}{k!}\;.
\end{equation}
\index{Taylor solution} Indeed $\d (s_x(T))=g(T)s_x(T)$. The
radius of convergence of $s_x(T)$ at $x$ is, by the usual
definition,
\begin{equation}\label{eq : liminf}
\index{Ray(M,x)@$Ray(M,x)$}
\liminf_k(|g_{[k]}(x)|/|k!|)^{-\frac{1}{k}}\;.
\end{equation}
\begin{definition}
We set
\begin{equation}
\omega\;:=\;|p|^{\frac{1}{p-1}}\;<\;1\;.
\end{equation}
\end{definition}
\begin{definition}\label{eq:radius}
The radius of convergence of $M$ at $\rho\in I$ is
\begin{eqnarray}\label{eq : radius of conv}
\index{Ray(M,rho)@$Ray(M,\rho)$}
Ray(M,\rho)&:=&\min\Bigl(\rho\;,\;\liminf_k(|g_{[k]}|_\rho/|k!|)^{-1/k}\Bigr)\nonumber\\
&=&
\min\Bigl(\rho\;,\;\omega\bigl[\limsup_k(|g_{[k]}|_\rho)^{1/k}\bigr]^{-1}\Bigr)\;.
\end{eqnarray}
We say that $M$ is solvable at $\rho$ if 
$Ray(M,\rho)=\rho$.
\end{definition}
This number represents the minimum radius of 
convergence of a solution at an unspecified point 
$x$ of norm $|x|=\rho$. More precisely there exists a 
complete field extension $\Omega/K$ 
and a point $t_\rho\in\Omega$, with 
$|t_\rho|=\rho$, such that for all $g\in\a_K(I)$ one has 
$|g|_\rho=|g(t_\rho)|_\Omega$. Such a point is called a 
$\rho$-\emph{generic point} (cf. \cite{Ch-Ro}). We deduce that
\begin{equation}
Ray(M,\rho)\;=\;
\min(\;\rho\;,\;\min_{|x|=\rho,\;x\in\Omega}\{\textrm{Radius of }s_x(T)\}\;)\;.
\end{equation}
Indeed this follows from \eqref{eq : liminf} 
and from the fact that 
$|g_{[k]}|_\rho=
\max_{|x|=\rho,\;x\in\Omega}|g_{[k]}(x)|_\Omega=
|g_{[k]}(t_\rho)|$.

\begin{remark}
The second equality of \eqref{eq : radius of conv} 
follows from the fact that the 
sequence $|k!|^{1/k}$ is convergent to $\omega$, and 
$|g_{[k]}|_\rho^{1/k}$
is bounded by $\max(|g_{[1]}|_\rho,\rho^{-1})$. The presence of $\rho$
in the minimum makes this definition invariant under change of
basis in $M$.
\end{remark}
If now $\d-g(T)$ is a differential operator with 
$g(T)\in\mathcal{E}_K$, then \eqref{eq : radius of conv} 
has a meaning for $\rho=1$ and it is an invariant by base 
changes of $M$. 

\begin{remark}\label{remark : properties}
We shall recall the following facts, that will be 
systematically used in the sequel:
\begin{enumerate}
\item If $M$ is a differential module over 
$\mathcal{E}_K$, then Definition \ref{eq:radius} 
has a meaning for $\rho=1$;
\item If $M$ is a differential module over $\a_K(I)$, and 
if $I$ is not reduced to a point, then the function 
$\rho\mapsto Ray(M,\rho)$ has the following properties
\begin{enumerate}
\item It is continuous on $I$.
\item It is piecewise of the form $\alpha\rho^\beta>0$, 
which is usually quoted as the \emph{$\log$-affinity 
property} (this means that the function $r\mapsto 
\log(Ray(M,\exp(r)))=\log(\alpha)+\beta r$ is affine).
\item The slopes $\beta$ are natural numbers.
\end{enumerate}
\item Recall that for all differential module $M,N$ one has
\begin{equation}\label{eq : tensor product radius}
Ray(M\otimes N,\rho)\;\geq\;
\min(Ray(M,\rho),Ray(N,\rho))
\end{equation} 
and equality holds if
$Ray(M,\rho)\neq Ray(N,\rho)$ 
(cf. \cite[Remark 1.2]{Rk1}). Notice that 
if for a given $\rho$ we have 
$Ray(M,\rho)= Ray(N,\rho)$, it
often happens that $Ray(M,\rho')\neq Ray(N,\rho')$ 
holds in a neighborhood of $\rho$ 
with the individual exception of $\rho$, 
so \emph{by continuity} we deduce that \eqref{eq : 
tensor product radius} is an equality also at 
$\rho$.
\item The $p$-th ramification $f(T)\mapsto f(T^p)$ is a 
$K$-linear ring endomorphism of $\mathcal{E}_K$ and 
of $\a_K(I)$ which is called (somehow improperly) 
\emph{Frobenius map}. We denote it by $\varphi$. 
By extension of scalars one can define an exact 
endo-functor which is called pull-back by Frobenius, 
denoted by $\varphi^*$ 
(cf. \cite{Astx}, \cite[1.2.3, 1.2.4]{Rk1}). 
The functor associates to a differential equation 
$\partial-g(T)$ the differential equation 
$\partial_T-p\cdot g(T^p)$. This is a technical tool 
of the theory used mainly to ``\emph{move the radii}'' 
of convergence of a differential module. More precisely if 
$M$ is a differential module over $\a_K(I^p)$, 
then for all $\rho\in I$ one has
\begin{equation*}
Ray(\varphi^*(M),\rho)\geq
\rho\cdot \min\Bigl(\Bigl(\frac{Ray(M,\rho^p)}{\rho^p}\Bigr)^{1/p}\;,\;
|p|^{-1} \frac{Ray(M,\rho^p)}{\rho^p}\Bigr),
\end{equation*}
and equality holds if $Ray(M,\rho^p)\neq\omega^p\rho^p$ 
(cf. \cite[Thm.7.2]{Astx}, \cite[10.3.2]{Kedlaya-book}).

If $Ray(M,\rho^p)>\omega\rho^p$, 
it is known that the functor can 
be (improperly speaking) ``\emph{inverted}'', this 
means that there exists a differential module $N$ such 
that $\varphi^*(N)\cong M$, and that 
such a module is unique (for a more precise statement 
see \cite[Thm. 7.5]{Astx}, 
\cite[10.4.2]{Kedlaya-book}). 
We say that $N$ is an 
\emph{antecedent by Frobenius of $M$}. 
\end{enumerate}

We refer to \cite{Rk1}, for the proof of these sentences 
and for all further properties and definitions.
\end{remark}
\section{Criterion of solvability for differential equations over \protect{$\mathcal{E}_K$}}
\label{section crit of solv} In this section we obtain a criterion
of solvability for differential equations over $\mathcal{E}_K$.
After a technical part (cf. Proposition 
\ref{division of the problem over Amice}), 
the main result will be actually an immediate 
consequence of the Lemma
\ref{criteria of solvability lemma2}.

\begin{lemma}[Small radius]\label{small radius2}
Let $\d-g(T)$, $g(T)\in\mathcal{E}_K$. Then
$Ray(\d-g(T),1)<\omega$ if and only if $|g(T)|_1>1$. 
In this case we have
\begin{equation}
Ray(\d-g(T),1)\;=\;
\omega\cdot|g(T)|_1^{-1}\;.
\end{equation}
\end{lemma}
\begin{proof}
See \cite[Lemma 1.1]{Rk1}.
\end{proof}

\subsection{Technical results}
There is no domain of the affine line where all the power 
series in $\mathcal{E}_K$ converge. 
If $\M$ is a differential module associated with 
the operator $\d-g$, with $g\in\mathcal{E}_K$, 
it is useful to have a basis of $\M$ 
in which $g$ converges on some domain.
For this, for all functions 
$g(T)=\sum_{i\in\mathbb{Z}}a_iT^i$ 
we set 
$g^-(T):=\sum_{i\leq -1}a_iT^i$, and 
$g^+(T):=\sum_{i\geq 1}a_iT^i$. 
The following proposition expresses any solvable $\M$ as 
tensor product of some 
solvable differential modules defined over a disk 
centered at $0$ and a disk centered at $\infty$. 
The ``\emph{Step} $4$'' of the proof is due to 
G.Christol.

\begin{proposition}
\label{division of the problem over Amice}
Let $\d-g(T)$, $g(T)\in\mathcal{E}_K$, be an 
equation which is solvable at $\rho=1$.
Then $\d-g^-(T)$, $\d-a_0$, and $\d-g^+(T)$ are all 
solvable at $\rho=1$.
\end{proposition}
\begin{proof}\emph{---Step 1:} 
By \eqref{s_x(T)}, the equation $\d-g^{-}(T)$ (resp.
$\d-g^{+}(T)$) has a convergent solution at $\infty$ (resp. at
$0$), hence $Ray(\d-g^{-}(T),\rho)=\rho$, for large values of
$\rho$ (resp. $Ray(\d-g^{-}(T),\rho)=\rho$, for $\rho$ close to
$0$). On the other hand, a direct computation proves that there is a $R^0>0$ such that
$Ray(\d-a_0,\rho)=R^0\cdot\rho$, for all $\rho$. Let
\begin{eqnarray}
R^-&:=&Ray(\d-g^{-}(T),1)\;,\\
R^+&:=&Ray(\d-g^{+}(T),1)\;,\\
R^0&:=&Ray(\d-a_0,1)\;.
\end{eqnarray}
We have to prove that $R_0=R^-=R^+=1$.\\

\emph{---Step 2:} We begin by proving that  
$R^+=R^-$, 
and that $R^0\geq R^-=R^+$. In the 
following picture $R:=R^-=R^+$, and for all operators $L$, 
we let $r:=\log(\rho)$ and 
$R(L,r):=\log(Ray(L,\rho)/\rho)$.

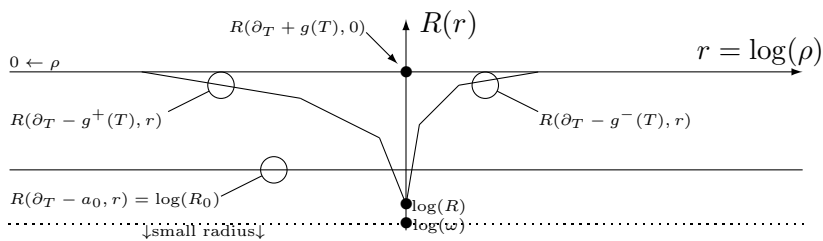
\begin{center}
\begin{picture}(300,80)
\put(150,0){\vector(0,1){80}} \put(0,60){\vector(1,0){300}}
\put(260,65){$r=\log(\rho)$} \put(155,75){$R(r)$}
\put(0,62){\begin{tiny}$0\leftarrow\rho$\end{tiny}}
\put(50,60){\line(6,-1){60}} 
\put(110,50){\line(2,-1){30}} 
\put(140,35){\line(2,-5){10}} 
\put(147.5,7.5){$\bullet$}\put(152,7.5){\tiny{$\log(R)$}}
\put(170,55){\line(6,1){30}} 
\put(170,55){\line(-1,-1){15}} 
\put(155,40){\line(-1,-6){5}} 
\put(0,23){\line(1,0){300}} 
\put(147.5,57.5){$\bullet$} 
\put(83,75){\begin{tiny}$R(\d+g(T),0)$\end{tiny}}
 \put(135,75){\vector(1,-1){12}}

 \put(80,55){\circle{10}}     
 \put(60,45){\line(2,1){15.5}}
 \put(0,40){\begin{tiny}$R(\d-g^+(T),r)$\end{tiny}}
 \put(180,55){\circle{10}}     
 \put(200,45){\line(-2,1){15.5}}
 \put(200,40){\begin{tiny}$R(\d-g^-(T),r)$\end{tiny}}
 \put(100,23){\circle{10}}     
 \put(80,13){\line(2,1){15.5}}

 \put(0,10){\begin{tiny}$R(\d-a_0,r)=\log(R_0)$\end{tiny}}
\put(147.5,0){$\bullet$} 
\put(152.5,0){\begin{tiny}$\log(\omega)$\end{tiny}}
\qbezier[100](0,2.5)(150,2.5)(300,2.5)
\put(50,-2){\begin{tiny}$\downarrow$small
radius$\downarrow$\end{tiny}}
\end{picture}
\end{center}

Since $\d-g$ is the tensor product of $\d-g^-$, 
$\d-g^+$, and $\d-a_0$, we deduce from point iii) of Remark \ref{remark : properties} that if two among 
$R^-,R^+,R^0$ are $1$, then the third is also equal to $1$.

Assume now by contrapositive that at least two among 
$R^-,R^+,R^0$ are strictly less than $1$. Then either 
$R^-<1$ or $R^+<1$. We want to prove that 
$R^+=R^-$, and that $R^0\geq R^-=R^+$. 

We assume for instance that $R^-<1$, the case 
where $R^+<1$ can be proved symmetrically. 

The function $r\mapsto R(\d-g^{-}(T),r)$ is concave, 
and $Ray(\d-g^-(T),1)=1$ if and only if the
slope of $r\mapsto R(\d-g^{-}(T),r)$ is $0$ for 
$r\to 0^+$.

The map $r\mapsto R(\d-g^{-}(T),r)$ for $r\to 0^+$ is 
strictly positive and the slope of $R(\d-a_0,r)=\log(R_0)$ 
is $0$. We deduce from point iii) of Remark \ref{remark : properties} that $Ray(\d-g^-(T),\rho)\neq Ray(\d-a_0,\rho)$ 
with the possible exception of an isolated $\rho$. 
Hence 
$Ray(\d-(a_0+g^-(T)),\rho)=
\min(Ray(\d-g^-(T),\rho),\rho R^0)$, for all
$\rho>1$ close to $1$. By continuity, this equality holds 
at $\rho=1$, that is 
\begin{equation}
Ray(\d-(a_0+g^-(T)),1)\;=\;\min(R^-,R^0)\;.
\end{equation}

Now since
$\d-g(T)$ is the tensor product of $\d-g^+(T)$ and
$\d-(a_0+g^-(T))$, and since $Ray(\d-g(T),1)=1$, we 
have again by point iii) of Remark 
\ref{remark : properties} that
\begin{equation}
R^+\;:=\;Ray(\d-g^+(T),1) \;=\;
Ray(\d-(a_0+g^-(T)),1)\;=\;\min(R^-,R^0)\;.
\end{equation}

We now claim that $R^0\geq R^-$, so the previous
equality implies $R^+=R^-$. Indeed if $R^->R^0$, 
then $R^+=R^0$.
Hence, as above, by concavity we deduce that 
for all $\rho<1$ one has $Ray(\d-g^{+}(T),\rho)\neq
Ray(\d-a_0,\rho)$, and that 
$Ray(\d-(a_0+g^{+}(T)),1)=R^0<R^-$. This implies 
$Ray(\d-g(T),1)=\min(R^0,R^-)=R^0<1$, 
contradicting the solvability of $\d-g$. Hence we must have $R^0\geq R^-=R^+$.\\

\emph{---Step 3:} If $R$ denotes the number 
$R^-=R^+$, then we have $R\geq \omega$. 
Indeed if $R^-<\omega$
or $R^+<\omega$, then, by \ref{small radius2}, $|g^-(T)|_1>1$ or
$|g^+(T)|_1>1$, hence $|g(T)|_1>1$ which is in contradiction with
the small radius lemma \ref{small radius2}, 
since the equation $\d-g(T)$ is solvable.\\

\emph{---Step 4:} We now prove that $R>\omega$. For 
this we need two lemmas:

\begin{lemma}[\protect{\cite[4.8.5]{Ch}}]\label{Katz}
Let $\d-g(T)$, $g(T)\in\mathcal{E}_K$, $|g(T)|_1\leq 1$ be some
equations. Then $Ray(\d-g(T),1)>\omega$ if and only if
$|g_{[s]}(T)|_1<1$, for some $s\geq 1$.\footnote{See 
Lemma \ref{q-Katz} for the q-analogue of this
lemma.}\hfill$\Box$
\end{lemma}

\begin{lemma}\label{|a_i|<1}
If $Ray(\d-g(T),1)>\omega$, where $g(T)=\sum a_iT^i$, then
$|a_i|<1$, for all $i\leq -1$.
\end{lemma}
\begin{proof} The matrix of $d/dT$ is $g_{[1]}:=g(T)/T$. By definition one
has
\begin{eqnarray*}
Ray(\d-g(T),1)=Ray(d/dT-g_{[1]}(T),1)&=&
\min\bigl(1,\liminf_s(|g_{[s]}(T)|_1/|s!|)^{-1/s}\bigr)\\
&=&\min\bigl(1,\omega\cdot\liminf_s(|g_{[s]}(T)|_1)^{-1/s}\bigr)\;,
\end{eqnarray*}
where $g_{[s]}(T)$ is associated to the derivation
$(\frac{d}{dT})^s$. Since $Ray(\d-g(T),1)>\omega$, 
hence $\lim_{s\to\infty}|g_{[s]}(T)|_1=0$. In particular
$|g_{[s]}(T)|_1<1$, for some $s\geq 1$. Moreover, by 
the small radius lemma \ref{small radius2}, 
we have $|g(T)|_1\leq 1$.
We proceed by contrapositive: let $-d$ be the smallest index such that $|a_{-d}|=1$. The
reduction of $g_{[1]}(T)=g(T)/T$ in $k(\!(t)\!)$ is of the form
$\overline{g_{[1]}(T)}=\overline{a}_{-d}t^{-d-1}+\cdots$. If
$-d\leq -1$, then an induction on the equation
$g_{[s+1]}=\frac{d}{dx}(g_{[s]})+g_{[s]}g_{[1]}$ shows that
$\overline{g_{[s]}(T)}=\overline{a}_{-d}^st^{(-d-1)s}+\cdots\neq
0$. This is in contradiction with the fact that
$|g_{[s]}(T)|_1<1$,
for some $s\geq 1$.
\end{proof}

Let us show now that $R>\omega$. Since $R^+=R^-=R$, it is
sufficient to show that $R^->\omega$. By Lemma \ref{|a_i|<1}, we have
$|a_i|<1$, for all $i\leq -1$. Since $\lim_{i\to-\infty}|a_i|=0$,
hence  $|g^-(T)|_1<1$. Then Lemma \ref{Katz} implies
$R^->\omega$.\\

\emph{---Step 5:} Since $R>\omega$, then we can take the
antecedent by Frobenius of $\d-g^-(T)$, $\d-g^+(T)$, $\d-a_0$.
More precisely, there exists $f^+(T)=\sum_{i\geq
0}b^+_iT^i\in\a([0,1[)^{\times}$, $f^-(T)=\sum_{i\leq
0}b^-_iT^i\in\a([1,\infty])^{\times}$, and there are functions
$g^{(1),-}(T)=\sum_{i\leq 0}a_i^{(1),-}T^i$,
$g^{(1),+}(T)=\sum_{i\geq 0}a_i^{(1),+}T^i$, $b_0\in K$ such that
\begin{eqnarray*}
pb_0&=&a_0+n\;,\quad\textrm{ for some } n\in\mathbb{Z}\;,\\
pg^{(1),-}(T^p)^\sigma&=&g^-(T)+\frac{\d(f^-(T))}{f^-(T)}\;,\\
pg^{(1),+}(T^p)^\sigma&=&g^+(T)+\frac{\d(f^+(T))}{f^+(T)}\;,
\end{eqnarray*}
where $\sigma:K\to K$ is an endomorphism of fields lifting of the $p$-th power map of $k$, and 
$(\sum a_i T^i)^\sigma$ means $\sum \sigma(a_i) T^i$.

We see immediately that $b_0^+\neq 0$ and $b_0^-\neq 0$, and that
$v_T(\d(f^+)/f^+)\geq 1$ and $v_{T^{-1}}(\d(f^-)/f^-)\geq 1$,
where $v_T$ is the $T-$adic valuation, and $v_{T^{-1}}$ is the
$T^{-1}-$adic valuation. Since $g^-(T)$ and $g^+(T)$ have no
constant term, we deduce that  $a_0^{(1),+}=0$ and
$a_0^{(1),-}=0$. Observe now that both $f^-$ and 
$f^+$ belong to $\mathcal{E}_K^{\times}$, hence
$\d-\bigl(g^{(1),-}(T)+b_0+g^{(1),+}(T)\bigr)$ is an 
antecedent of Frobenius of $\d-g(T)$, and it is then 
solvable.\\

\emph{--- Step 6:} Steps $1$, $2$, $3$, $4$ are still true for the
antecedent. In particular, if we set
\begin{eqnarray}
R^-(1)&:=&Ray(\d-g^{(1),-}(T),1)\;,\\
R^+(1)&:=&Ray(\d-g^{(1),+}(T),1)\;,\\
R^0(1)&:=&Ray(\d-b_0,1)\;,
\end{eqnarray}
then we must have $R^-(1)=R^{+}(1)>\omega$. Let
$R(1):=R^{-}(1)=R^{+}(1)$, then $R(1)=R^{1/p}$ by the property of
the antecedent. This implies $R>\omega^{1/p}$. 

Now the condition $R(1)>\omega$, guarantee the 
existence of the antecedent of the antecedent, and the 
process can be iterated indefinitely. This shows that
$R>\omega^{1/p^h}$ for all $h\geq 0$, that is $R=1$. 
\end{proof}

\begin{corollary}\label{g^+ is trivial}
We have $a_0 \in \mathbb{Z}_p$ and
$\d-g^+(T)$ is trivial. 
\end{corollary}
\begin{proof} 
By the transfer theorem, the Taylor solution at $0$ of
$\d-g^+(T)$ is convergent in the open unit disk. This 
solution is invertible with inverse the solution of the dual 
differential module, hence it is bounded and belongs to 
$\mathcal{E}_K$. 
\end{proof}
The following corollary, together with Corollary 
\ref{canonical ext over amice}, constitute
the analogue of the Katz's 
canonical extension functor \cite{Katz-Can}:
\begin{corollary}[Katz's canonical extension]
Let $M$ be a solvable rank one differential module over 
$\mathcal{E}_K$ 
represented in a basis by the operator $\d-g(T)$, with 
$g(T)=\sum_{i\in\mathbb{Z}}a_iT^i\in\mathcal{E}_K$.
Then there exists a basis of $M$ in which the associated 
operator is
\begin{equation}
\d-(a_0+g^-(T))\;.
\end{equation}
In particular $M$ comes by scalar extension from a 
differential module over the closed unit disk 
$D:=\{|T|\geq 1\}\cup\{\infty\}$ centered at $\infty$. 
It has a regular singularity at $\infty$ if and only if 
$a_0\in\mathbb{Z}$, and it has no singularities on $D$ 
otherwise. 
\end{corollary}


\subsection{Criterion of solvability}

Following \cite{Rk1} we now introduce an exponential 
series which is the solution of our differential equations. 
We refer to \cite{Rk1} for all notations and properties.

We set $\J:=\{n\in\mathbb{Z}\;|\; (n,p)=1, n\geq 1 \}$.

For all ring $A$ (not necessarily with unit element) we 
denote by $\W(A)$ the ring of $p$-typical Witt vectors 
of infinite length with coefficients in $A$.
 Its elements are sequences $\bs{a}=(a_0,a_1,\ldots)$ 
of elements of $A$. For all $m\geq 0$ we call 
\emph{phantom vector} of $\bs{a}$ the tuple 
$\phi_m(\bs{a}):=a_0^{p^m}+pa_1^{p^{m-1}}+\cdots+p^ma_m$. 
The map $\W(A)\to A^{\mathbb{N}}$ associating to 
$\bs{a}$ the tuple $(\phi_0(\bs{a}),\phi_1(\bs{a}),
\ldots)$ is a morphism of functors in rings. 
In order to make a more evident distinction between 
Witt vectors and phantom components, we denote Witt 
vector by the letter $\lambda$ and phantom components 
by the letter $\phi$, moreover we also use a bracket 
$\ph{\phi_0,\phi_1,\ldots}$ to indicate an element of the 
Ring $A^{\mathbb{N}}$.

Let now $A=T\O_K[[T]]$. We now recall some notions 
from \cite[Section 4.3]{Rk1}.
For all $\bs{\lambda}=(\lambda_0,\lambda_1,
\ldots)\in\W(K)$ and all integer $d>0$ we set
\begin{equation}
\bs{\lambda}T^d\;:=\;
(\lambda_0T^d,\lambda_1T^{pd},\lambda_2T^{p^2d},
\ldots)\;\in\;\W(TK[[T]])\;.
\end{equation}
To a sum 
$\sum_{d>0}\bs{\lambda}_dT^d\in\W(TK[[T]])$ we 
associate the following exponential of Artin-Hasse type
\begin{equation}
E(\sum_{d>0}\bs{\lambda}_dT^d,1)\;=\;
\prod_{d>0}\exp(\sum_{m\geq 
0}\phi_{d,m}\frac{T^{dp^m}}{p^m})
\end{equation}
where for all $d>0$ the tuple $(\phi_{d,0},\phi_{d,1},
\ldots)$ is the phantom vector of $\bs{\lb}_d$.
The map $\W(K)\to K^{\mathbb{N}}$ being an 
isomorphism, it easy to prove that any exponential of the 
form $\exp(\sum_{d>0}b_d\frac{T^d}{d})
\in 1+TK[[T]]$ 
can be uniquely decomposed as  
\begin{equation}
\exp(\sum_{d>0}b_d\frac{T^d}{d})\;=\;
\exp(\sum_{n\in\J}
\sum_{m\geq 0}b_{np^m}
\frac{T^{np^m}}{np^m})\;=\;
\exp(\sum_{n\in\J}
\sum_{m\geq 0}\phi_{n,m}
\frac{T^{np^m}}{p^m})\;=\;
E(\sum_{n\in\J}\bs{\lambda}_nT^n,1)\;,
\end{equation}
where $\phi_{n,m}=b_{np^m}/n$, 
and $\bs{\lb}_n\in\W(K)$ is defined as the unique Witt 
vector with phantom vector $(\phi_{n,0},\phi_{n,1},
\ldots)$. We refer to  \cite[Section 4.3]{Rk1} for further 
properties.

The following Lemma asserts that 
solutions of rank one solvable differential equations over 
the open unit disk are those exponentials as above whose 
Witt vectors have coefficients in $\O_K$.

\begin{lemma}\label{criteria of solvability lemma2}
The differential equation $\d-g^+(T)$, $g^+(T)=\sum_{i\geq
1}a_iT^i\in\a([0,1[)$ is solvable if and only if there exists a
family $\{\lb_{n}\}_{n\in\J}$, $\lb_n\in\W(\O_K)$, with phantom
components $\phi_{n}=(\phi_{n,0},\phi_{n,1},\ldots)$ satisfying
\begin{equation}\label{a_np^m=n phi_n,m}
a_{np^m}=n\phi_{n,m}\;,\qquad\textrm{ for all }n\in\J,\; m\geq
0\;.
\end{equation}
In other words, we have $\exp(\sum_{i\geq 1}a_i
\frac{T^i}{i})=E(\sum_{n\in\J}\lb_nT^n,1)$, where
\begin{equation}\label{E(sum_J lb_n T^n,1)}
E(\sum_{n\in\J}\lb_nT^n,1):=\exp(\sum_{n\in\J}\sum_{m\geq
0}\phi_{n,m}T^{np^m}/p^m)\;.
\end{equation}
\end{lemma}
\begin{proof} 
The formal series $E(\sum_{n\in\J}\lb_nT^{n},1)\in
1+T\O_K[[T]]$ is solution of the equation
$L:=\d-\sum_{n\in\J}\sum_{m\geq 0}n\phi_{n,m}T^{np^m}$. Since this
exponential converges in the unit disk, then $Ray(L,\rho)=\rho$,
for all $\rho<1$, and $L$ is solvable. 

Conversely, assume that $\d-g^+(T)$
is solvable. Then the Witt vectors
$\lb_n=(\lambda_{n,0},\lambda_{n,1},\ldots)$ are defined by the
relation \eqref{a_np^m=n phi_n,m}. For example, for all $n\in\J$ we
have
\begin{equation}
\label{expliciting lambda_i in function of a_i}
\lambda_{n,0} = \frac{a_{n}}{n} \;\;,\qquad \lambda_{n,1} =
\frac{1}{p}\left(\frac{a_{np}}{n} - \bigl(\frac{a_n}{n}\bigr)^p
\right)\;.
\end{equation}
We must show that $|\lambda_{n,m}|\leq 1$, for all $n\in\J$,
$m\geq 0$.

\emph{---Step 1:} By the small radius Lemma \ref{small radius2},
we have $|a_i|\leq 1$, for all $i\geq 1$. Hence, by
\eqref{expliciting lambda_i in function of a_i}, for all $n\in\J$,
we have $|\lambda_{n,0}|\leq 1$. Then the exponential
$$E(\sum_{n\in\J}(\lambda_{n,0},0,0,\ldots)T^{n},1)=
\exp(\sum_{n\in\J}\sum_{m\geq
0}\lambda_{n,0}^{p^m}\frac{T^{np^m}}{p^m})$$ converges in the unit
disk and is solution of the operator $Q^{(0)}:=\d - h^{(0)}(T)$,
where $h^{(0)}(T)=\sum_{n\in\J}\sum_{m\geq
0}n\lambda_{n,0}^{p^m}T^{np^m}$. So $Q^{(0)}$ is 
solvable.

\emph{---Step 2:} The tensor product operator $\d - (g^+(T) -
h^{(0)}(T))$ is again solvable and satisfies
$g^+(T)-h^{(0)}(T)=p\cdot g^{(1)}(T^p)$, for some $g^{(1)}(T)\in T
K[[T]]$. In other words, the ``antecedent by ramification''
$\varphi_p^*$ of the equation $\d - (g^+(T) - h^{(0)}(T))$ is
given by $\d-g^{(1)}(T)$, which is then solvable.

\emph{---Step 3: } We observe that $g^{(1)}(T) \!=\!\frac{1}{p}
\sum_{n\in\J}\sum_{m\geq
0}(a_{np^{m+1}}-n(\frac{a_n}{n})^{p^{m+1}}) T^{np^{m}}$, and again
by the small radius lemma, we have
$|a_{np}\!-n(\frac{a_n}{n})^{p}|\leq |p|$, 
which implies $|\lambda_{n,1}|\leq 1$. 

The process can be iterated indefinitely. This proves that 
$|\lambda_{n,m}|\leq 1$ for all $n,m$.
\end{proof}

\begin{remark}\label{discussion}We shall now consider the general
case of an equation $\d-g(T)$, with
$g(T)=\sum_{i\in\mathbb{Z}}a_iT^i\in\mathcal{E}_K$, and get a
criterion of solvability. Suppose that $\d-g(T)$ is solvable. We
know that $\d-g^-(T)$, $\d-a_0$ and $\d-g^+(T)$ are all solvable
(cf. \ref{division of the problem over Amice}). We can then
consider $\d-g^-(T)$ as an operator on $]1,\infty]$ (instead of
$[1,\infty]$), and the precedent lemma \ref{criteria of
solvability lemma2} give us the existence of a family of Witt
vector $\{\lb_{-n}\}_{n\in\J}\subset\W(\O_K)$, satisfying
$a_{-np^m} = -n\phi_{-n,m}$, for all $n\in\J$, and all $m\geq 0$.
Conversely, suppose given two families $\{\lb_{-n}\}_{n\in\J}$ and
$\{\lb_{n}\}_{n\in\J}$, with $\lb_n\in\W(\O_K)$. Since the phantom
components of $\lb_n$ are bounded by $1$, then $|a_i|$ is bounded
by $1$, and then $g^+(T)$ belongs to $\mathcal{E}_K$.

What we need now is a condition on the family 
$\{\lb_{-n}\}_{n\in\J}$
in order that the series 
\begin{equation}
g^-(T):=\sum_{n\in\J}\sum_{m\geq 0}-n
\phi_{-n,m}T^{-np^m}
\end{equation} 
belongs to $\mathcal{E}_K$.
\end{remark}

\begin{proposition}\label{criteria for belong to Amice}
Let $\{\lb_{-n}\}_{n\in\J}$, $\lb_{-n}\in\W(\O_K)$, be a family of
Witt vectors. Let $\ph{\phi_{-n,0},\phi_{-n,1},\ldots}$ be the
phantom vector of
$\lb_{-n}:=(\lambda_{-n,0},\lambda_{-n,1},\ldots)$. The series
$$g^{-}(T):=\sum_{n\in\J}\sum_{m\geq 0}-n\phi_{-n,m}T^{-np^m}\;,$$
belongs to $\mathcal{E}_K$ if and only if
\begin{equation}\label{lb<1 and lim_n lb_n,m =0}
\left\{\begin{array}{rcl}
|\lambda_{-n,m}|<1&,& \textrm{ for all }n\in\J\;,\;\textrm{ for all }m\geq 0\;;\\
&&\\
\lim_{n\in\J,n\to\infty}\lambda_{-n,m}=0&,&\textrm{ for all }
m\geq 0\;,
\end{array}\right.
\end{equation}
as in the picture
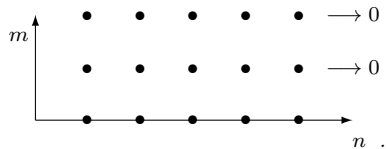
\begin{center}
\begin{scriptsize}
\begin{picture}(120,60)
\put(0,10){\vector(0,1){40}} \put(0,10){\vector(1,0){120}}
\put(-10,40){$m$}\put(120,0){$n$}\put(130,0){.}
\put(17.5,8.25){$\bullet$} \put(37.5,8.25){$\bullet$}
\put(57.5,8.25){$\bullet$} \put(77.5,8.25){$\bullet$}
\put(97.5,8.25){$\bullet$}
\put(17.5,27.5){$\bullet$} \put(37.5,27.5){$\bullet$}
\put(57.5,27.5){$\bullet$} \put(77.5,27.5){$\bullet$}
\put(97.5,27.5){$\bullet$}
\put(17.5,47.5){$\bullet$} \put(37.5,47.5){$\bullet$}
\put(57.5,47.5){$\bullet$} \put(77.5,47.5){$\bullet$}
\put(97.5,47.5){$\bullet$}
\put(110,47.5){$\longrightarrow 0$}
\put(110,27.5){$\longrightarrow 0$}
\end{picture}
\end{scriptsize}
\end{center}
\end{proposition}
We need the following lemma:
\begin{lemma}\label{phi_j to 0 Leftrightarrow lambda_j<1}
Let $\lb=(\lambda_{0},\lambda_{1},\ldots)\in\W(\O_K)$ be a Witt
vector, and let $\ph{\phi_0,\phi_1,\ldots}\in\O_K^{\mathbb{N}}$ be
its phantom vector. Then $\phi_j\to 0$ in $\O_K$ if and only if
$|\lambda_j|<1$, for all $j\geq 0$.
\end{lemma}
\begin{proof} 
The set of Witt vectors whose phantom components
go to $0$ is clearly an ideal $I\subset\W(\O_K)$ 
containing
$\W(\frak{p}_K)$, where $\frak{p}_K$ is the 
maximal ideal of $\O_K$. 
Reciprocally, suppose $\phi_j\to 0$, since
$\phi_j=\lambda_0^{p^j}+p\lambda_1^{p^{j-1}}+\cdots+p^j\lambda_j$,
we have $|\lambda_0|<1$. Then
$\lb^{(1)}:=(0,\lambda_1,\lambda_2,\ldots)=\lb-(\lambda_0,0,\ldots)$
lies again in the ideal $I$, and hence $\phi_j(\lb^{(1)})=
p\lambda_1^{p^{j-1}}+ p^{2}\lambda_2^{p^{j-2}}+
\cdots+p^j\lambda_j\to 0$. This shows that 
$|\lambda_1|<1$. Proceeding inductively one sees that 
$|\lambda_j|<1$, for all $j\geq 0$. \end{proof}

We now are ready to give the proof of Proposition 
\ref{criteria for belong to Amice} : 
\begin{proof} 
Assume that
$g^-(T)=\sum_{n\in\J}\sum_{m\geq 0}-n\phi_{-n,m}T^{-np^m}$ lies in
$\mathcal{E}_K$. This happens if and only if $\lim_{np^m\to
\infty}\phi_{-n,m}=0$, and implies $\lim_{m\to
\infty}\phi_{-n,m}=0$ for all $n\in\J$. By Lemma 
\ref{phi_j to 0
Leftrightarrow lambda_j<1}, we have $|\lambda_{-n,m}|<1$, for all
$n\in\J$ and all $m\geq 0$. An easy induction shows that
$\lim_{n\in\J,n\to\infty}\lambda_{-n,m}=0$, for all $m\geq 0$.

Reciprocally, assume that $\{\lb_{-n}\}_{n\in\J}$ satisfies the
condition \eqref{lb<1 and lim_n lb_n,m =0}. We must show that
$\lim_{np^m\to\infty}\phi_{-n,m}=0$. For all $\varepsilon>0$, we
choose $k\geq 0$ such that $|p^{k+1}|< \varepsilon$. By
assumption, for all $0\leq m\leq k$, there exists $N_m$ such that
$|\lambda_{-n,m}|< \varepsilon$, for all $n\geq N_m$. Let
$N:=\max(N_0,\ldots,N_k)$. Then
$$\phi_{-n,m}=\underbrace{\lambda_{-n,0}^{p^m}+
\cdots+p^{k}\lambda_{-n,k}^{p^{m-k}}}_{< \varepsilon\;,\textrm{ if
} n\geq N}+
\underbrace{p^{k+1}\lambda_{-n,k+1}^{p^{m-k-1}}+\cdots+
p^m\lambda_{-n,m}}_{<\varepsilon}\;.$$ Hence
$|\phi_{-n,m}|<\varepsilon$, if $n\geq N$. On the other hand, by
assumption, there is $\delta<1$ such that
$|\lambda_{-n,m}|\leq\delta<1$, for all $m=0,\ldots,k$,
$n=0,\ldots,N$. Hence there exists $M$ such that
$|\lambda_{-n,0}^{p^m}|,\ldots,|\lambda_{-n,k}^{p^{m-k}}|<\varepsilon$,
for all $m\geq M$. Then $|\phi_{-n,m}|\leq \varepsilon$, for all
$n\geq N$, $m\geq M$. Hence
$\lim_{np^m\to\infty}\phi_{-n,m}=0$.\end{proof}

\begin{definition}\label{Conv}
We denote by $\mathrm{Conv}(\mathcal{E})$
\index{Conv@$\mathrm{Conv}(\mathcal{E})$}the set of families
$\{\lb_{-n}\}_{n\in\J}$, with
$\lb_{-n}=(\lambda_{-n,0},\lambda_{-n,1},\ldots)\in\W(\O_K)$,
satisfying condition \eqref{lb<1 and lim_n lb_n,m =0}.
\end{definition}

\begin{corollary}[Criterion of solvability]\label{criterion of solv over amice}
The equation $\d-g(T)$,
$g(T)=\sum_{i\in\mathbb{Z}}a_iT^i\in\mathcal{E}_K$, is solvable if
and only if $a_0\in\mathbb{Z}_p$, and there exist two families
$\{\lb_{-n}\}_{n\in\J}$ and $\{\lb_{-n}\}_{n\in\J}$ with
$\lb_{-n},\lb_{n}\in\W(\O_K)$, for all $n\in\J$, such that
$\{\lb_{-n}\}_{n\in\J}\in\mathrm{Conv}(\mathcal{E}_K)$, and
moreover
\begin{equation}
a_{np^m}=n\phi_{n,m}\quad,\quad a_{-np^m}=-n\phi_{-n,m}\;.
\end{equation}
In other words, its formal solution
$T^{a_0}\exp(\sum_{i\leq-1}a_i\frac{T^i}{i}) \exp(\sum_{i\geq
1}a_i\frac{T^i}{i})$ can be represented by the symbol
\begin{equation}\label{formal solution}
 T^{a_0}\cdot\exp(\sum_{n\in\J}\sum_{m\geq
0}\phi_{-n,m}\frac{T^{-np^m}}{p^m})\cdot\exp(\sum_{n\in\J}\sum_{m\geq
0}\phi_{n,m}\frac{T^{np^m}}{p^m})\;,
\end{equation}
where $(\phi_{-n,0},\phi_{-n,1},\ldots)$ (resp.
$(\phi_{n,0},\phi_{n,1},\ldots)$) is the phantom vector of
$\lb_{-n}$ (resp. $\lb_n$), and hence
\begin{equation}\label{matrix of a solvable differential equation}
\qquad\qquad
g(T)\;=\;\sum_{n\in\J}\sum_{m\geq
0}-n\phi_{-n,m}T^{-np^m}+a_0+\sum_{n\in\J}\sum_{m\geq
0}n\phi_{n,m}T^{np^m}\;.\qquad\qquad\Box
\end{equation}
\end{corollary}

\begin{corollary}[Katz's canonical extension]
\label{canonical ext over amice}
Let
$\d\textrm{-}\mathrm{Mod}(\a_K([1,\infty]))^{\mathrm{sol}}_{rk=1}$
be the category of rank one differential modules over
$\a_K([1,\infty])$, solvable at all $\rho\geq 1$, with a regular
singularity at $\infty$ (i.e. the matrix of $\d$ converge at
$\infty$ and hence belongs to $\a_K([1,\infty])$). The scalar
extension functor
$$\d\textrm{-}\mathrm{Mod}(\a_K([1,\infty]))^{\mathrm{sol}}_{rk=1} \longrightarrow
\d\textrm{-}\mathrm{Mod}(\mathcal{E}_K)^{\mathrm{sol}}_{rk=1}$$ is
an equivalence.
\end{corollary}
\begin{proof} 
Corollary \ref{criterion of solv over amice} shows that
gives a correspondence between the objects. Indeed, all differential
equations $\d-g(T)$ over $g(T)=g^-(T)+a_0+g^+(T)\in\mathcal{E}_K$
is isomorphic over $\mathcal{E}_K$ to the equation
$\d-(g^-(T)+a_0)$. On the other hand, let
$M,N\in\d-Mod(\a_K([1,\infty]))^{\mathrm{sol}}_{rk=1}$, and let
$\d-g_M$, $\d-g_N$ be the operators corresponding to a chosen
basis of $M$ and $N$. An element of $\Hom(M,N)\stackrel{\sim}{\to}
M^{\vee}\otimes N$ is then a solution of the operator
$\d-(g_N-g_M)$. This solution will be of the form
$y(T)=T^{a_0}\exp(\sum_{n\in\J}\sum_{m\geq
0}\phi_{-n,m}T^{-np^m}/p^m)$, for some $\phi$. Since we are
supposing that this solution belongs to $\a_K([1,\infty])$, then
$a_0\in\mathbb{Z}$ and this exponential lies in
$\a_K([1,\infty])$. Since the same argument works for
$\Hom_{\d}(M\otimes\mathcal{E}_K,N\otimes\mathcal{E}_K)$, and
since $\a_K([1,\infty])\subset\mathcal{E}_K$, then the map
$\Hom_{\d}(M,N)\to\Hom_{\d}(M\otimes\mathcal{E}_K,N\otimes\mathcal{E}_K)$
is bijective. 
\end{proof}

\begin{remark}\label{missing morphism}
We are not able to obtain a complete description of the
isomorphism class of a given differential equation over
$\mathcal{E}_K$. Namely, over the Robba ring 
$\mathfrak{R}_K$, we know that a solution of a
differential equation lies in $\mathfrak{R}_K$ if and only 
if the
corresponding Witt vector 
(in a convenient basis of $M$) satisfies a certain property 
\cite[Theorem 3.1]{Rk1}. But we do not have the 
analogous
result over $\mathcal{E}_K$. In other words, we do not 
have a
necessary and sufficient condition on the Witt vector
$\sum_{n\in\J}\lb_{-n}T^{-n}$ in order that
$E(\sum_{n\in\J}\lb_{-n}T^{-n},1)$ belongs to $\mathcal{E}_K$.
\end{remark}

\section*{\textsc{Second part : 
solvability of rank one $q$-difference equations over 
$\mathcal{E}_K$}}
\addcontentsline{toc}{section}{\textsc{Second part : 
solvability of rank one $q$-difference equations over 
$\mathcal{E}_K$}}

We shall establish the $q-$analogue of the results of section
\ref{section crit of solv}. In order to do that, we will need some
numerical lemmas (cf. section \ref{numerical lemmas}) and a
generalization of the result of E.Motzkin (cf. \cite{Motz}, and
section \ref{The motzkin decomposition} below).
As a consequence we will prove that for 
$|q-1|<\omega$ we have an equivalence of 
$q$-confluence as in \cite{Pu-q-Diff}.
 
We shall point out that, almost all statements are true for
$|q-1|<1$. The only obstructions to obtain the 
confluence in the case $\omega\leq|q-1|<1$ are 
\begin{enumerate}
\item the
existence of the ``antecedent by Frobenius'' (used in 
``Step $5$'' of Proposition 
\ref{q-division of the problem over Amice}), which 
is proved in
\cite{DV-Dwork} only for $|q-1|<\omega$;
\item the ``\emph{Step }$0$'' of Theorem 
\ref{q-criteria of solvability lemma}. 
\end{enumerate}
Namely, the existence of an antecedent by Frobenius 
holds with $|q-1|<1$ over the Robba Ring, but 
the proof uses the Confluence \cite{Pu-q-Diff}. It is 
reasonable to conjecture that a more direct proof is 
possible generalizing \cite{DV-Dwork} to the case 
$|q-1|<1$. The author hopes
that these difficulties will be overcoming in future. 

For these reasons the hypothesis $|q-1|<\omega$ will be
introduced systematically starting from 
\ref{hypothesis q-1<omega}
on. Before Hypothesis \ref{hypothesis q-1<omega}
 we will suppose 
that $|q-1|<1$.

\section{Some numerical Lemmas}\label{numerical lemmas}
\begin{lemma}\label{coefficients of logarithm}
Let us fix an integer $j\geq 0$. If $j\geq 1$ we assume 
$\omega^{1/p^{j-1}}<\rho<\omega^{1/p^{j}}$, and 
if $j=0$ we assume $\rho<\omega$. Then
\begin{equation}\label{eq : condition j t}
\frac{\rho^{p^j}}{|p^j|}\;>\; 
\sup(\rho^{r}/|r|\;:\; r\geq 1,r\neq p^j) \;.
\end{equation}
Moreover, we have
\begin{equation}
\rho<\frac{\rho^p}{|p|}<\cdots<\frac{\rho^{p^{j-1}}}{|p^{j-1}|}<\frac{\rho^{p^j}}{|p^j|}\quad;\quad
\frac{\rho^{p^j}}{|p^j|}>\frac{\rho^{p^{j+1}}}{|p^{j+1}|}>\frac{\rho^{p^{j+2}}}{|p^{j+2}|}>\cdots\;.
\end{equation}
\end{lemma}
%
%
%
\begin{proof}
If $r\neq p^k$, for all $k\geq 0$, then $|r|=|p|^{v}$, 
with $v:=v_p(r)$, hence 
$\rho^{r}/|r| < \rho^{p^{v}}/|p|^{v}$. 
This proves \eqref{eq : condition j t}. 
Now the condition $\rho^{p^{k-1}}/|p^{k-1}|<\rho^{p^k}/|p^k|$ is
equivalent to $\rho_1<\frac{\rho_1^p}{|p|}$, where
$\rho_1:=\rho^{p^{k-1}}$, and it 
is verified if and only if
$\rho_1>\omega$, that is $\rho>\omega^{\frac{1}{p^{k-1}}}$. 
On the other hand,
the inequality $\rho^{p^{k-1}}/|p^{k-1}|>\rho^{p^k}/|p^k|$ is
equivalent to $\rho<\omega^{\frac{1}{p^{k}}}$.
\end{proof}

\begin{lemma}\label{coefficients of exponential}
Let $n\geq 1$ be a natural number. Let $l(n):=[\log_p(n)]$, where
$[x]$ denotes the greatest integer smaller than or equal to the real
number $x$. Then for all $k\geq n$ we have
\begin{equation}
\left|\frac{k!}{n!}\right|^{\frac{1}{k-n}}\;\geq\;|p|^{l(n)+1}\;.
\end{equation}
In particular, if $c\leq |p|^{l(n)+1}$, 
then for all $k\geq n$ we have
\begin{equation}
\label{rho^n/n! geq rho^k/k! for all k geq n}
\frac{c^n}{|n!|}\;\geq\;
\frac{c^k}{|k!|}\;.
\end{equation}
\end{lemma}
\begin{proof} 
If $k=n$, the relation is trivial; suppose $k>n$. The
equation \eqref{rho^n/n! geq rho^k/k! for all k geq n} is equivalent
to $c\leq|\frac{k!}{n!}|^{\frac{1}{k-n}}$. Since
$|n!|=\omega^{n-S_n}$, where $S_n$ is the sum of the 
digits of the base $p$ expansion of $n$, then
$|\frac{k!}{n!}|^{\frac{1}{k-n}}=\omega^{1+\frac{S_n-S_k}{k-n}}$.
If $n=n_0+n_1p+n_2p^2+\cdots+n_{l(n)}p^{l(n)}$, with $0\leq n_i\leq p-1$,
then $S_n=n_0+n_1+\cdots +n_{l(n)}$, hence $1\leq S_n\leq (p-1)(l(n)+1)$.
This shows that
\begin{equation}
1+\frac{S_n-S_k}{k-n}\leq 1+\frac{(p-1)(l(n)+1)-1}{k-n}\leq 1+
(p-1)(l(n)+1) -1 =(p-1)(l(n)+1)\;.
\end{equation}
Hence $|\frac{k!}{n!}|^{\frac{1}{k-n}}\geq
\omega^{(p-1)(l(n)+1)}=|p|^{l(n)+1}$, for all $k>n$. 
\end{proof}

\begin{definition}\label{Def : q^a}
Let $q\in K$ be such that $|q-1|<1$. 
For all complete 
valued field extension $\Omega/K$, and all 
$\alpha\in\Omega$ we define
\begin{equation}
q^\alpha:=((q-1)+1)^\alpha:=
\sum_{k\geq
0}\binom{\alpha}{k}(q-1)^k\;,
\end{equation}
where $\tbinom{\alpha}{k}:=\frac{\alpha(\alpha-1)(\alpha-2)\cdots(\alpha-k+1)}{k!}$.
\end{definition}
If $|\alpha|>1$, then
$|\tbinom{\alpha}{k}|=\frac{|\alpha|^k}{|k!|}$, hence $q^\alpha$
converges exactly for $|q-1|<\omega/|\alpha|$. 

If $|\alpha|\leq 1$, then $q^\alpha$ converges at least for
$|q-1|<\omega$, in particular if 
$\alpha\in\mathbb{Z}_p$, then $q^{\alpha}$ converges 
at least for $|q-1|<1$. For a detailed
discussion on the radius of convergence of $q^\alpha$ 
see \cite[Ch.IV, Prop.7.3]{DGS}.

\begin{lemma}\label{(q^a-1)/(q-1) --> a}
Let $\alpha\in\Omega$ and $q\in K$ be as in Definition 
\ref{Def : q^a}. Then
\begin{equation}
\lim_{q\to 1}\frac{q^\alpha-1}{q-1}=\alpha\;.
\end{equation}
\end{lemma}
\begin{proof} Write $\frac{(q^\alpha-1)}{(q-1)}=\frac{((q-1)+1)^\alpha-1}{(q-1)}=
\alpha+\sum_{k\geq
2}\tbinom{\alpha}{k}(q-1)^{k-1}$. 
Let $s:=\max(|\alpha|,1)$, and for all $n\geq 1$ let
$l(n):=[\log_p(n)]$.
We now prove that if $|q-1|\leq |p|^{l(2)+1}/s$, 
then for all $k\geq 2$ we have 
$|\tbinom{\alpha}{k}(q-1)^{k-1}|\leq
|\tbinom{\alpha}{2}(q-1)|$ which is enough to conclude. 

Assume $k\geq n\geq 1$. The condition 
 $|\tbinom{\alpha}{k}(q-1)^{k-1}|\leq
|\tbinom{\alpha}{n}(q-1)^{n-1}|$ is equivalent to 
\begin{equation}\label{4444}
|q-1|\leq |\tbinom{\alpha}{n}/\tbinom{\alpha}{k}|^{\frac{1}{k-n}}=
\left(\frac{|k!|}{|n!|}
\frac{1}{|(\alpha-n)\cdots(\alpha-k+1)|}\right)^{\frac{1}{k-n}}\;.
\end{equation}
By Lemma \ref{coefficients of exponential} we know that $(\frac{k!}{n!})^{\frac{1}{k-n}}\geq
|p|^{l(n)+1}$. On the other hand, it is clear that
$|(\alpha-n)\cdots(\alpha-k+1)|\leq s^{k-n}$. 
Hence the right hand side of \eqref{4444} is bigger than $|p|^{l(n)+1}/s$. The claim is proved.\end{proof}

\begin{lemma}\label{valuation of q^d-1}
Let $j\geq 0$. 
If $j=0$, assume that $|q-1|<\omega$, and if $j\geq 1$ we assume that $\omega^{1/p^{j-1}}<|q-1|<\omega^{1/p^j}$. Let $d:=\alpha
p^m\in\mathbb{Z}_p$, with $\alpha\in\mathbb{Z}_p$ such that $(\alpha,p)=1$. Let
$i:=\min(m,j)$. Then
\begin{equation}
|q^d-1|=|d|\cdot\frac{|q-1|^{p^i}}{|p|^i}=|p^{m-i}||q-1|^{p^i}\;.
\end{equation}
\end{lemma}
\begin{proof} Since $(\alpha,p)=1$, hence
$\left|\binom{\alpha}{1}\right|=1$. Then
\begin{equation}
|q^\alpha-1|=|((q-1)+1)^\alpha-1|=|
\sum_{k=1}^\infty\tbinom{\alpha}{k}(q-1)^k|=|q-1|\;.
\end{equation}
Moreover, one has
$|q^{\alpha p^m}-1| = |((q^\alpha-1)+1)^{p^m}-1| =
|\sum_{k=1}^{p^m}\tbinom{p^m}{k}(q^\alpha-1)^k|$. 
Since for all $k\leq p^m$ one has $|\binom{p^m}{k}|=\frac{|p|^m}{|k|}$, we deduce
$|\binom{p^m}{k}(q^\alpha-1)^k|=|p^m|\frac{\rho^k}{|k|}$. The claim follows from 
Lemma \ref{coefficients of logarithm} applied to
$\rho=|q-1|=|q^\alpha-1|$.
\end{proof}

\section{The Motzkin decomposition}\label{The motzkin 
decomposition}
In \cite{Motz} a decomposition theorem for analytic 
element over
an affino\"id domain of the line 
(i.e. a set of type
$\mathbb{P}_K^1-\cup_{i=1,\ldots,n}\mathrm{D}^-_K(a_i,r_i)$) is
proved. In \cite{Ch-Motz} G.Christol generalizes this
decomposition for matrices with coefficients in analytic
functions. We now generalize that theorem for series in
$\mathcal{E}_K$ (cf. \ref{motzkin}).\\

Let $I\subseteq\mathbb{R}_{\geq 0}$ be any non empty interval. 
We set $I_0:=I\cup [0,\rho]$ (resp. 
$I_\infty:= I\cup[\rho,+\infty]$), where $\rho\in I$. As 
an example if $I=[r_1,r_2[$ then $I_0=[0,r_2[$ and 
$I_\infty=[r_1,+\infty]$.
\begin{theorem}\label{Thm : Motzkin annulus}
Let $I\subseteq\mathbb{R}_{\geq 0}$ be any interval. 
Then each invertible function $a(T)\in\a_K(I)^\times$ 
can be uniquely written as
\begin{equation}\label{eq : deco - interval}
a(T)\;=\;\lambda\cdot  T^N \cdot a^-(T)\cdot a^+(T)\;,
\end{equation}
where $\lambda\in K$, $N\in\mathbb{Z}$, 
$a^+(T)=1+\alpha_1T+\alpha_2T^2+\cdots
\in 1+T\a_K(I_0)^\times$ and 
$a^-(T)=1+\alpha_{-1}T^{-1}+\alpha_{-2}T^{-1}+
\cdots \in 1+T^{-1}\a_K(I_\infty)^\times$.
\end{theorem}
Before giving the proof we need two lemmas.
Let $\overline{I}$ be the closure of $I$ in 
$\mathbb{R}$. Invertible functions are bounded, so it 
has a meaning to consider their norm $|.|_\rho$ for all 
$\rho\in\overline{I}$.
\begin{lemma}\label{|a_-i|<1}
Let  
$a^+(T)=1+\alpha_{1}T+\alpha_{2}T^{2}+\cdots$ 
be an invertible function in $\a_K(I_0)$. 
If $r\in I_0$, for all $i\geq 1$ we have 
$|\alpha_i|r^i< 1$. 
If $r\in \overline{I_0}$ for all $i\geq 1$ we have 
$|\alpha_{i}|r^{i}\leq 1$.

The same claim holds for functions $a^-(T)\in 
\a_K(I_\infty)$.
\if{
Let $a^-(T)=1+\alpha_{-1}T^{-1}+\alpha_{-2}T^{-2}+\cdots$ 
be an invertible function in $1+T^{-1}\a_K([r,\infty])$. 
Let $a^+(T)=1+\alpha_{1}T+\alpha_{2}T^{2}+\cdots$ 
be an invertible function in $\a_K([0,r[)$. Then
for all $i\geq 1$ we have 
$|\alpha_i|r^i\leq 1$ and 
$|\alpha_{-i}|r^{-i}<1$.
}\fi
\end{lemma}
\begin{proof} 
By replacing $T$ with $\gamma_r T$, with 
$|\gamma_r|=r$, we
can suppose $r=1$. 

Since $a^+$ is invertible, its valuation polygon has no 
breaks (cf. \cite[Chapitre 2]{Ch-Ro}), so 
for all $\rho\leq 1$ we have 
$|a^+|_\rho=|a^+(0)|=1$. 
Hence $|\alpha_i|\leq 1$ for all $i\geq 1$.

If now $r=1\in I_0$, and if there exists $i\geq 1$ such that 
$|\alpha_{i}|=1$, the reduced series 
$\overline{a^+(T)}\in k[T]$ is a non constant polynomial. 
The zeros of $\overline{a^+(T)}$ lift into zeros of $a^+(T)$, 
which contradicts the fact that $a^+(T)$ is invertible, 
hence without zeros in the closed unit disks.
\if{Let now $X=T^{-1}$, and let
$P(X):=a^-(X^{-1})=1+\alpha_{-1}X+\alpha_{-2}X^2+\cdots$. 
As for $a^+$ for all $i\geq 1$ 
we have $|\alpha_{-i}|\leq 1$. 

If now there exists $i\geq 1$ such that 
$|\alpha_{-i}|=1$, the reduced series 
$\overline{P(X)}\in k[X]$ is not constant. 
The zeros of $\overline{P(X)}$ lift into zeros of $P(X)$, 
which contradicts the fact that $P(X)$ is invertible, 
hence without zeros in the closed unit disks.
}\fi
\if{
Then there exists a root $\bar{\alpha}$ of
$\overline{P^-(X)}$ in $k^{\mathrm{alg}}$. This shows that there is
$\zeta$ in $K^{\mathrm{alg}}$, with $|\zeta|=1$ and such that
$|P(\zeta)|<1$. If $\zeta$ is a zero of $P^-(X)$, then $P^-(X)$ is not
invertible, this is in contradiction with the hypothesis. Hence
$\zeta$ is not a zero of $P^-(X)$. Let us write $P(X)=\sum_{i\geq
0}b_i(X-\zeta)^i$. Since $|P^-(\zeta)|<1$, then $|b_0|<1$. Since
$|P^-(X)|_1=1$, then $\lim_{\rho\to 1}\sup_{i\geq 0}|b_i|\rho^i=1$.
This implies that the polygon of valuation of $|P^-(X)|$ 
has a change of slope. Hence $P^-(X)$ has a zero in the closed unit ball
$|T|\leq 1$. This is the desired contradiction.
}\fi
\end{proof}
\begin{lemma}\label{c_0=1}
Let $\rho\in\overline{I}$. Let
$a^-(T)=1+\alpha_{-1}T^{-1}+\alpha_{-2}T^{-2}+\cdots\in
\a_K(I_\infty)^{\times}$, and
$a^+(T)=1+\alpha_{1}T+\alpha_{2}T^{2}+\cdots\in
\a_K(I_0)^{\times}$ be invertible functions. 
Then
\begin{equation}
|a^-(T)\cdot a^+(T)-1|_\rho < 1\;.
\end{equation}
\end{lemma}
\if{
\begin{lemma}\label{c_0=1}
Let $r_1\leq r_2$ and $\rho\in [r_1,r_2]$. Let
$a^-(T)=1+\alpha_{-1}T^{-1}+\alpha_{-2}T^{-2}+\cdots\in
\a_K([r_1,\infty])^{\times}$, and
$a^+(T)=1+\alpha_{1}T+\alpha_{2}T^{2}+\cdots\in
\a_K([0,r_2[)^{\times}$ be invertible functions. 
Then
\begin{equation}
|a^-(T)\cdot a^+(T)-1|_\rho < 1\;.
\end{equation}
\end{lemma}
}\fi
\begin{proof} Write $a^-(T)a^+(T)=\sum_{n\in\mathbb{Z}}c_nT^n$. If we 
define $\alpha_0:=1$, then, for all $n\geq 0$ one has 
$c_n=\sum_{k=0}^\infty
\alpha_{n+k}\alpha_{-k}$, and
$c_{-n}=\sum_{k=0}^{\infty}\alpha_{-n-k}\alpha_k$. By
Lemma \ref{|a_-i|<1}, either  
for all $k\geq 1$ we have 
$|\alpha_{-k}|\rho^{-k}<1$, and $|\alpha_k|\rho^k\leq 1$, or for all $k\geq 1$ we have 
$|\alpha_{-k}|\rho^{-k}\leq 1$, and $|\alpha_k|\rho^k< 1$. Since $\lim_{k\to\pm\infty}|\alpha_{k}|\rho^{k}=0$ 
then for all $n\geq 1$ one have 
$|c_{n}|\rho^n<1$, and
$|c_{-n}|\rho^{-n}<1$, and $|c_0-1|< 1$. 
\end{proof}
\begin{proof}[Proof of Theorem \ref{Thm : Motzkin annulus}] We first prove the claim for a rational 
fraction $a=P/Q$, $P,Q\in K[T]$. 
Let $Z_0$ and $V_0$ (resp. $Z_\infty$ and $V_\infty$) 
be the set of its zeros and poles respectively whose 
valuation belongs to $I_0$ (resp. $I_\infty$). Since 
$\mathrm{Gal}(K^{\mathrm{alg}}/K)$ acts by 
isometric maps, the  polynomials 
$P_0       :=\prod_{z\in Z_0-\{0\}}(T-z)$, 
$P_\infty :=\prod_{z\in Z_\infty}(T-z)$, 
$Q_0      :=\prod_{v\in V_0-\{0\}}(T-v)$, 
$Q_\infty:=\prod_{v\in V_\infty}(T-v)$ lie in $K[T]$ 
since their coefficients are invariant by Galois. 
Now $P=\alpha T^s P_0P_\infty$ and 
$Q=\beta T^r Q_0Q_\infty$, for convenient 
$\alpha,\beta\in K$, $r,s\in\mathbb{N}$. 
We then have $a^+(T)=\alpha' P_\infty/Q_\infty$, 
$a^-(T)=\beta'P_0/Q_0$, for convenient constants 
$\alpha',\beta'\in K$.

We now deduce \emph{by density} 
the case where $I$ is a compact interval. If $\|.\|_I$ 
is the sup-norm on $\{|T|\in I\}$, the Frechet 
topology of $\a_K(I)$ is given by the individual norm 
$\|.\|_I$, and $(\a_K(I), \|.\|_I)$ is a Banach algebra.

Let $a(T)=\sum_{i\in\mathbb{Z}}b_iT^i$ be 
as in the claim. For all $\rho\in I$ we have 
$\lim_{i\to\pm\infty}|b_i|\rho^i=0$ so for all 
$\rho\in I$ we can consider the integer
$N_\rho:=\min(i\;|\;|b_i|\rho^i=|a(T)|_\rho)$. 
Since $a$ is invertible, the $\log$-function 
$r\mapsto\log(|a(T)|_{\exp(r)})$ is affine 
on $\overline{I}$ of slope $N\in\mathbb{Z}$. 
This means that 
$N_\rho=N$ for all $\rho\in\overline{I}-\inf(I)$. 
Moreover if 
$\inf(I)\in I$ the equality also holds at $\rho=\inf(I)$ by 
\cite[Thm. 5.4.7]{Ch-Ro}. Multiplying by 
$(b_NT^N)^{-1}$ we can assume $N=0$ and 
$|a|_\rho=1$ for all $\rho\in\overline{I}$.

Let $a_n(T)$ be a sequence of rational fractions 
convergent to $a(T)$. 
Then for $n$ sufficiently large $a_n(T)$ has no poles
nor zeros on $\{|T|\in I\}$, hence $a_n(T)$ admits such 
a decomposition:
$a_n(T)=\lambda_nT^{N_n}a^-_n(T)a_n^+(T)$. 
Moreover there exists $n_0$ 
such that for all $n\geq n_0$ we have $N_n=0$, and 
$|\lambda_n|=1$. 

We now prove that, if $a^+_n=1+h_n^+$ and 
$a^-_n=1+h_n^-$, then for all $n,m\geq n_0$ the 
norms $|\lambda_n-\lambda_m|$, $\|a_n^+-a_m^+\|_I=\|h_n^+-h_m^+\|_I$, and $\|a_n^--a_m^-\|_I=\|h_n^--h_m^-\|_I$ are all bounded by $\|a_n-a_m\|_I$. Since $T^{-1}\a_K([r_1,\infty])$ and
$T\a_K([0,r_2])$ are closed sets in $\a_K(I)$, 
this will be enough to show that the sequences 
$n\mapsto\lambda_n$, $n\mapsto h_n^-$, and $n\mapsto h_n^+$, all converge in $K$, $T^{-1}\a_K([r_1,\infty])$ and
$T\a_K([0,r_2])$ respectively. 
This will be enough to obtain the desired decomposition 
\eqref{eq : deco - interval}.

Let $n,m\geq n_0$. We let 
$1+h^-:=\frac{1+h_n^-}{1+h_m^-}$ and 
$1+h^+:=\frac{1+h_m^+}{1+h_n^+}$. Then
\begin{eqnarray}
\| a_n-a_m\|_I\;=\;
\| \lambda_na_n^-a_n^+-\lambda_ma_m^-a_m^+
\|_I
&=&\Bigl\| \frac{\lambda_na_n^-a_n^+-
\lambda_ma_m^-a_m^+}{a_n^+a_m^-}\Bigr\|_I\\
&=&\Bigl\|\lambda_n\frac{a_n^-}{a_m^-}-\lambda_m\frac{a_m^+}{a_n^+} \Bigr\|_I\\
&=&\|(\lambda_n-\lambda_m)+
\lambda_nh^--\lambda_mh^+\|_I
\end{eqnarray}
We now notice that $h^-$ (resp. $h^+$) 
is a power series of the form 
$b_{-1}T^{-1}+b_{-2}T^{-2}+\cdots$  
(resp. $b_{1}T+b_{2}T^{2}+\cdots$), hence
for all $\rho\in I$ we have
\begin{eqnarray}
|(\lambda_n-\lambda_m)+\lambda_nh^-
-\lambda_mh^+|_\rho&\;=\;&
\max(|\lambda_n-\lambda_m|,
|\lambda_n|\sup_{i\leq -1}|b_i|\rho^i,
|\lambda_m|\sup_{i\geq 1}|b_i|\rho^i)\\
&\;=\;&
\max(|\lambda_n-\lambda_m|, |h^-|_\rho, 
|h^+|_\rho)\;.
\end{eqnarray}
So we find
\begin{eqnarray}
\|a_n-a_m\|_I\;=\;
\sup(|\lambda_n-\lambda_m|,\|h^-\|_I,\|h^+\|_I)\;.
\end{eqnarray}
Now $\|h^+\|_I=\Bigl\| 
\frac{1+h_m^+}{1-h_n^+}-1
\Bigr\|_I=\Bigl\| 
\frac{h_m^+-h_n^+}{1-h_n^+}
\Bigr\|_I=\|h_m^+-h_n^+\|_I$, and analogously 
$\|h^-\|_I=\|h_m^--h_n^-\|_I$. This gives the desired 
inequalities.

The case where $I$ is non compact is deduced 
by expressing $I$ as increasing union of compact 
intervals $J_n\subset J_{n+1}\subset I$. 
The uniqueness of the decomposition shows that the 
decomposition over $J_n$ coincides with that over 
$J_{n+1}$, and we conclude.
\end{proof}
\begin{theorem}\label{motzkin}
Assume that $K$ is discretely valuated. 
Let $a(T)\in\mathcal{E}_K$. Then
there exist  $\lambda\in K$,  $N\in\mathbb{Z}$, $a^-(T)=1+h^-(T)$
invertible in $1+T^{-1}\a_K([1,\infty])$, with $h^-(T)=\sum_{i\leq
-1}\alpha_iT^i$, and $a^+(T)=1+h^+(T)$ invertible in
$1+T\a_K([0,1[)$, with $h^+(T)=\sum_{i\geq 1}\alpha_iT^i$, such
that
$$a(T)=\lambda\cdot T^N \cdot a^-(T)\cdot a^+(T).$$
Moreover, such a decomposition is unique.
\end{theorem}
\begin{proof}
The claim can not be deduced immediately ``by density'' 
because rational fractions are not dense in 
$\mathcal{E}_K$ with respect to the Gauss norm 
$|.|_1$. 
However the claim holds for functions in $\Ed_K$ 
because they converge on some 
annulus.\footnote{Actually rational fractions are dense
in $\Ed_K$ with respect to the $\mathcal{LF}$ topology 
induced by the Robba ring $\mathfrak{R}_K$.}
Now $\Ed_K$ is dense in
$\mathcal{E}_K$ with respect to the Gauss norm. 

The assumption $K$ discretely valued arises now to 
prove that 
$\inf\{i\in\mathbb{Z},
\textrm{ such that }|b_i|=|a(T)|_1\}$ 
is not equal to $+\infty$. This guarantee the existence of 
$N<+\infty$. 
We can now reproduce the same proof as 
Theorem \ref{Thm : Motzkin annulus} replacing 
$\|.\|_I$ by the Gauss norm $|.|_1$. We obtain the 
desired decomposition.
\end{proof}
\begin{remark}
As already mentioned, if the functions converge in some 
appropriate domains, the above results extend to 
matrices \cite{Ch-Motz}, \cite[Thm.6.5]{Astx}. 
We do not know whether such a generalization exists for 
matrices with coefficients in $\mathcal{E}_K$. 
The main applications from our point of view would be 
the study of differential equations with coefficients in 
that ring.
\end{remark}
\section{Criterion of solvability for $q$-difference equations over $\mathcal{E}_K$}
\label{q-crit of solv}
\begin{hypothesis}\label{K is discrete valued}
From now on the valuation on $K$ will be discrete valuation, in
order to have theorem \ref{motzkin}.
\end{hypothesis}

We denote by
\begin{eqnarray}
&&\sigma_q:f(T)\mapsto
f(qT)\;,\qquad\partial_T:=T\frac{d}{dT}\;,\\
&&d_q:=\frac{\sigma_q-1}{(q-1)T}\;,\;\;\quad\qquad\Delta_q:=\frac{\sigma_q-1}{q-1}
\;.
\end{eqnarray}
Let $A$ be one of the rings $\mathfrak{R}_K$, 
$\mathcal{E}_K$, $\Ed_K$, $\a_K(I)$.
A $q$-difference equation is finite free $A$-module $M$
together with an automorphism $\sigma_q:M\simto M$ 
satisfying $\sigma_q(am)=\sigma_q(a)\sigma_q(m)$ for 
all $a\in A$, $m\in M$. This corresponds in a basis of 
$M$ to an expression of 
the form $\sigma_q(Y)=a(q,T)Y$, where 
$a(q,T)\in GL_n(A)$.

From the action of  $\sigma_q:
M\simto M$ we can define the action of $d_q$ and 
$\Delta_q$ on $M$. In a basis of $M$ the action of 
$d_q$ amounts to an equation of the form 
$d_q(Y)=g_{[1]}(q,T)Y$, with 
$g_{[1]}(q,T)=\frac{a(q,T)-\mathrm{Id}}{q-1}\in M_n(A)$. As for differential equations 
we can attribute to such a module a radius of 
convergence. Namely the formal solution is given by
\begin{equation}
Y_q(T,t)\;:=\;\sum_{s\geq 0}g_{[s]}(q,t)\cdot
\frac{(T-t)_{q,s}}{[s]_q^!}
\end{equation}
where for all natural $n\geq 0$
\begin{equation}
[n]^!_q\;:=\; 
\frac{\prod_{i=1}^n(q^i-1)}{(q-1)^n}
\end{equation}
is the $q-$factorial, and $(T-t)_{q,s}:=
(T-t)(T-qt)\cdots(T-q^{s-1}t)$ 
(cf. \cite{Pu-q-Diff} for more 
details), and $g_{[s]}$ is the matrix of the action of 
$d_q^n$ on $M$. Namely $g_[0]=\mathrm{Id}$, 
$g_{[1]}=\frac{a(q,T)-1}{(q-1)T}$, and for all 
$s\geq 2$ one has $g_{[s+1]}(q,T)=
d_q(g_{[s]}(q,T))+\sigma_q(g_{[s]}(q,T))\cdot 
g_{[1]}(q,T)$.

The radius of convergence of $d_q-g_{[1]}$ 
is then defined as
\begin{equation}
Ray(d_q-g_{[1]}(q,T),\rho)\;:=\;
\min(\liminf_s(|g_{[s]}(q,T)|_\rho/[s]_q^!)^{-1/s},\rho)
\end{equation}
This number is attached to the operator $d_q-g_{[1]}$, 
but it is not invariant by base changes of $M$.
The radius is always less than or equal to $\rho$, 
if it is equal to $\rho$ we say that $\sigma_q-a(q,T)$ is 
\emph{solvable} at $\rho$. If $A=\mathcal{E}_K$ and $\rho=1$ we simply say 
\emph{solvable} (without specifying $\rho=1$). 
If $A=\mathfrak{R}_K$, we say that the equation is 
solvable if 
$\lim_{\rho\to 1^-}Ray(\sigma_q-a(q,T),\rho)=1$.

\subsection{Preliminary lemmas}
\begin{lemma}\label{limit of q-factorial}
Assume $|q-1|<1$. 
Then the sequence $|[n]^!_q|^{1/n}$ converges to a 
real number strictly less than $1$, we 
call $\omega_q < 1$ that number.
Moreover, let $\kappa$ be the smallest integer such that
$|q^\kappa-1|<\omega$, then
$$\omega_q=\left\{\begin{array}{ccl}
\omega &\textrm{ if }& \kappa=1\;,\\
(|\frac{q^\kappa-1}{q-1}|\cdot\omega)^{\frac{1}{\kappa}} &\textrm{
if }& \kappa\geq 2\;.
\end{array}\right.$$
\end{lemma}
\begin{proof} \cite[3.5]{DV-Dwork}. \end{proof}

\begin{lemma}\label{d_q(f) leq k! f}
Let $|q-1|<1$. For all $f(T)\in\a_K(I)$, for all
$\rho\in I$ and all $k\geq 1$, we have
$|\frac{d_q^k}{[k]^!_q}(f)|_\rho\leq
\rho^{-k}|f|_\rho$. The same
result is true for $f\in\mathcal{E}_K$ and $\rho=1$.
\end{lemma}
\begin{proof}\cite[2.1]{DV-Dwork}.\end{proof}
\begin{lemma}[$q$-Small Radius, $q-$analogue of 
Lemma \ref{small radius2}]
\label{q-Young}
Let $q\in K$, $|q-1|<1$, and let 
$I\subseteq\mathbb{R}_{\geq 0}$ be any interval. 
Let $\sigma_q-a(q,T)$, $a(q,T)\in\a_K(I)$ be some 
rank one $q$-difference equation. Let
$R_\rho:=Ray(\sigma_q-a(q,T),\rho)$  be the radius of convergence
of the equation at $\rho\in I$. Then
\begin{equation}
\label{R_rho geq ....}
R_\rho\geq\frac{\omega_q}{\max(|g_{[1]}(q,T)|_\rho,\rho^{-1})}=\frac{\omega_q\cdot\rho\cdot|q-1|}{\max(|a(q,T)-1|_\rho,|q-1|)}
\end{equation}
Moreover $R_\rho<\omega_q\cdot\rho$ if and only if
$|a(q,T)-1|_\rho>|q-1|$, and in this case 
\begin{equation}\label{R_rho = .... small radius}
R_\rho = \frac{\omega_q\cdot \rho\cdot |q-1|}{|a(q,T)-1|_\rho}\;.
\end{equation}
The same assertions hold for solvable $q$-difference equations
over $\mathcal{E}_K$, with $\rho=1$.
\end{lemma}
\begin{proof} Let $g_{[s]}(q,T)\in\a_K(I)$ be the $1\times 1$ matrix of
$(d_q)^s$. By definition
\begin{eqnarray}\label{q-radius explicit formula}
Ray(d_q-g_{[1]}(q,T),1)&=&
\min\bigl(\rho\;,\;\liminf_s(|g_{[s]}(q,T)|_1/|[s]^!_q|)^{-\frac{1}{s}}\bigr)\nonumber\\
&=&\min\bigl(\rho\;,\;\omega_q\cdot\liminf_s(|g_{[s]}(q,T)|_1)^{-\frac{1}{s}}\bigr)\;.\quad\qquad
\end{eqnarray}
One has inductively $|g_{[s]}|_\rho \leq
\max(|g_{[1]}|_\rho,\rho^{-1})^s$, this shows \eqref{R_rho geq
....}. Moreover, if $|g_{[1]}|_\rho>\rho^{-1}$, then
$|g_{[s]}|_\rho = \max(|g_{[1]}|_\rho,\rho^{-1})^s$ and \eqref{R_rho
= .... small radius} holds. Reciprocally, if
$R_\rho<\omega_q\cdot\rho$, then, by \eqref{R_rho geq ....}, one has
$|a(q,T)-1|> |q-1|$. \end{proof}

\begin{lemma}\label{lambda=1 and N=0}
Let $|q-1|<1$. Let $\sigma_q-a(q,T)$ be a
rank one \emph{solvable} equation such that $a(q,T)\in\mathfrak{R}_K$ or
$a(q,T)\in\mathcal{E}_K$. Let $a(q,T)=\lambda_q
T^{N}a^-(q,T)a^+(q,T)$ be the Motzkin decomposition of $a(q,T)$
(cf. Theorems \ref{Thm : Motzkin annulus}, \ref{motzkin}), then $N=0$ and $|\lambda_q-1|<1$.
\end{lemma}
\begin{proof} 
The solvability implies $|a(q,T)-1|_1\leq |q-1|<1$ (cf.
Lemma \ref{q-Young}), hence $|a(q,T)|_1=1$. More precisely, with the
notations as in the proof of Lemma \ref{c_0=1}, one has
$|\lambda_q\sum_{n\in\mathbb{Z}}c_nT^{n+N}-1|_1\leq|q-1|<1$. We
know that $\sup_{n\neq 0}|c_n|<1$ and $|c_0-1|<1$ (cf. Lemma
\ref{c_0=1}). If $N\neq 0$, then $|\lambda_q c_0T^N|_1< 1$ and
$|\lambda_q c_{-N}-1|<1$. The first one implies $|\lambda_q|<1$,
which contradicts the second one. Hence $N=0$. 
We deduce that $|\lambda_q
c_0-1|<1$ which implies $|\lambda_q-1|<1$.
\end{proof}

\begin{lemma}\label{q-radius of constant}
Let $|q-1|<1$. There exists $R_0>0$ such that
$Ray(\sigma_q-\lambda_q,\rho)=R_0\cdot\rho$, for all
$\rho\in[0,\infty[$.
\end{lemma}
\begin{proof} 
By \cite[1.2.4]{DV-Dwork}, one has
$$\left|g_{[n]}(T)\right|_\rho^{\frac{1}{n}}=
\frac{|\sum_{j=0}^{n}(-1)^j\binom{n}{j}_{q^{-1}}q^{\frac{-j(j-1)}{2}}
\lambda_q^j|^{1/n}}{|q-1|\cdot\rho}\;.$$ Since the numerator does
not depend on $\rho$, the lemma is proved. 
\end{proof}

\subsection{The settings}
\label{the settings}

As for differential equations, we shall find a description of the
formal solution of a given solvable $q-$difference equation
\begin{equation}\label{q-diff equation}
\sigma_q(y_q) = a(q,T)\cdot y_q\;,
\end{equation}
with $a(q,T)\in\mathcal{E}_K$. We will show that solutions of
$q-$difference equations are actually solutions of 
differential equation of the form \eqref{formal solution}. By Lemma \ref{lambda=1 and N=0},
we know that
\begin{equation}
a(q,T)=\lambda_q\cdot a^-(q,T)\cdot a^+(q,T)\;,
\end{equation}
with $a^-(q,T):=1+\sum_{i\leq -1} \alpha_iT^i$, and
$a^+(q,T):=1+\sum_{i\geq 1} \alpha_iT^i$. Now write formally
\begin{equation}
a^-(q,T):=\exp(\sum_{i\leq -1}a_iT^i)\;,\qquad
a^+(q,T):=\exp(\sum_{i\geq 1}a_iT^i)\;.
\end{equation}
Then the formal solution of \eqref{q-diff equation} is
\begin{equation}\label{formal solution explicit}
y_q(T):=\exp\bigl(\sum_{i\leq -1}\frac{a_i}{q^i-1}T^i\bigr)\cdot
q^{a_0}\cdot\exp\bigl(\sum_{i\geq 1}\frac{a_i}{q^i-1}T^i\bigr)\;.
\end{equation}
We are interested to study this exponential in the case in which
the equation \eqref{q-diff equation} is solvable. The 
main result will be the Criterion of solvability 
\ref{criterion of solvability for q-difference}.

\subsection{Technical results} %
In this section $q\in\mathrm{D}^-(1,1)$ is fixed. We will 
omit the index $q$ in the series. The following 
proposition is the 
q-analogue of Proposition 
\ref{division of the problem over Amice} for the Robba 
ring.

\begin{proposition}
\label{q-division of the problem over Robba} Let
$|q-1|<1$. Let $\sigma_q-a(T)$, $a(T)=\lambda a^-(T)
a^+(T) \in \mathfrak{R}_K$ be a solvable equation. Then $\sigma_q-a^-(T)$,
$\sigma_q-\lambda$, $\sigma_q-a^+(T)$ are all solvable.
\end{proposition}
\begin{proof} 
With analogous notations of Proposition \ref{division of the problem over
Amice}, we find the following picture:
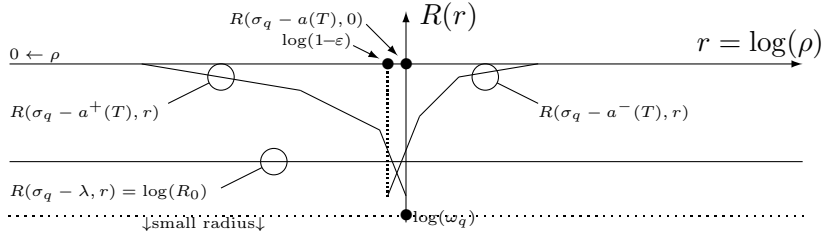
\begin{center}
\begin{picture}(300,80)
\put(150,0){\vector(0,1){80}} \put(0,60){\vector(1,0){300}}
\put(260,65){$r=\log(\rho)$} \put(155,75){$R(r)$}
\put(0,62){\begin{tiny}$0\leftarrow\rho$\end{tiny}}
\put(50,60){\line(6,-1){60}} 
\put(110,50){\line(2,-1){30}} 
\put(140,35){\line(2,-5){10}} 
\put(170,55){\line(6,1){30}} 
\put(170,55){\line(-1,-1){15}} 
\put(155,40){\line(-2,-5){12}} 
\put(0,23){\line(1,0){300}} 
\put(147.5,57.5){$\bullet$} 
\put(83,75){\begin{tiny}$R(\sigma_q-a(T),0)$\end{tiny}}
 \put(135,75){\vector(1,-1){12}}

\put(140.5,57.5){$\bullet$} 
\put(103,67){\begin{tiny}$\log(1\!\!-\!\!\varepsilon)$\end{tiny}}
\put(130,67){\vector(2,-1){10}}
\qbezier[30](143,10)(143,35)(143,60)
 \put(80,55){\circle{10}}     
 \put(60,45){\line(2,1){15.5}}
 \put(0,40){\begin{tiny}$R(\sigma_q-a^+(T),r)$\end{tiny}}
 \put(180,55){\circle{10}}     
 \put(200,45){\line(-2,1){15.5}}
 \put(200,40){\begin{tiny}$R(\sigma_q-a^-(T),r)$\end{tiny}}
 \put(100,23){\circle{10}}     
 \put(80,13){\line(2,1){15.5}}

 \put(0,10){\begin{tiny}$R(\sigma_q-\lambda,r)=\log(R_0)$\end{tiny}}
\put(147.5,0){$\bullet$} 
\put(152.5,0){\begin{tiny}$\log(\omega_q)$\end{tiny}}
\qbezier[100](0,2.5)(150,2.5)(300,2.5)
\put(50,-2){\begin{tiny}$\downarrow$small
radius$\downarrow$\end{tiny}}
\end{picture}
\end{center}
Since there exists a common interval 
$I:=]1-\varepsilon,1[$ in
which all operators exist, and since the slope of
$Ray(\sigma_q-a^-,\rho)$ 
(resp. $Ray(\sigma_q-a^+,\rho)$) is
strictly positive (resp. negative) in $I$, hence there are at 
most $3$ points in which these graphics cross. 
Hence, by continuity, for all $\rho\in I$ one has 
\begin{equation}\label{ray=min}
Ray(\sigma_q-a,\rho)=\min(\;Ray(\sigma_q-a^-,\rho)\;,\;Ray(\sigma_q-a^+,\rho)\;,\;Ray(\sigma_q-\lambda,\rho)\;)\;.
\end{equation}
By assumption $\lim_{\rho\to
1^-}Ray(\sigma_q-a,\rho)=1$, hence the claim follows. \end{proof}
We now give the 
q-analogue of Proposition 
\ref{division of the problem over Amice} for the ring $\mathcal{E}_K$:
\begin{proposition}
\label{q-division of the problem over Amice}
Let $|q-1|<\omega$. Let $\sigma_q-a(T)$, $a(T)=\lambda a^-(T)
a^+(T) \in \mathcal{E}_K$, be a solvable equation. Then
$\sigma_q-a^-(T)$, $\sigma_q-\lambda$, $\sigma_q-a^+(T)$ are all
solvable.
\end{proposition}
\begin{proof} \label{proof of first reduction} Steps $1$ and $2$ of this
proof coincide with the same steps of the proof of Proposition \ref{division
of the problem over Amice}. We will expose it without
proofs for fixing notation. The first part of this proposition
does not use the hypothesis $|q-1|<\omega$, so we will assume this
hypothesis starting from Hypothesis 
\ref{hypothesis q-1<omega}.

--- \emph{Step 1 : }By \cite[3.6]{DV-Dwork}, the equation
$\sigma_q-a^{-}(T)$ (resp. $\sigma_q-a^{+}(T)$) has a convergent
solution at $\infty$ (resp. at $0$), hence
$Ray(\sigma_q-a^{-}(T),\rho)=\rho$, for large values of $\rho$
(resp. $Ray(\sigma_q-a^{-}(T),\rho)=\rho$ for $\rho$ close to
$0$). Let $R^0$ be as in Lemma \ref{q-radius of constant},
\begin{eqnarray}
R^-&:=&Ray(\sigma_q-a^{-}(T),1)\;,\\
R^+&:=&Ray(\sigma_q-a^{+}(T),1)\;.
\end{eqnarray}

--- \emph{Step 2 : We have $R^+ = R^-$ and $R^0 \geq R^- = R^+$ (as in
the following picture in which $R := R^- = R^+$)}.\\

We set $r:=\log(\rho)$, and
$R(r):=\log(Ray(\sigma_q-a(T),\rho)/\rho)$.
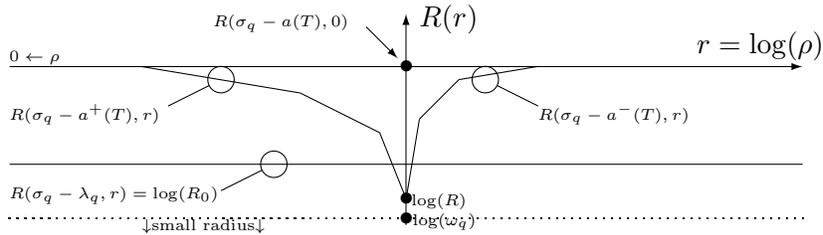
\begin{center}
\begin{picture}(300,80)
\put(150,0){\vector(0,1){80}} \put(0,60){\vector(1,0){300}}
\put(260,65){$r=\log(\rho)$} \put(155,75){$R(r)$}
\put(0,62){\begin{tiny}$0\leftarrow\rho$\end{tiny}}
\put(50,60){\line(6,-1){60}} 
\put(110,50){\line(2,-1){30}} 
\put(140,35){\line(2,-5){10}} 
\put(147.5,7.5){$\bullet$}\put(152,7.5){\tiny{$\log(R)$}}
\put(170,55){\line(6,1){30}} 
\put(170,55){\line(-1,-1){15}} 
\put(155,40){\line(-1,-6){5}} 
\put(0,23){\line(1,0){300}} 
\put(147.5,57.5){$\bullet$} 
\put(77,75){\begin{tiny}$R(\sigma_q-a(T),0)$\end{tiny}}
 \put(135,75){\vector(1,-1){12}}

 \put(80,55){\circle{10}}     
 \put(60,45){\line(2,1){15.5}}
 \put(0,40){\begin{tiny}$R(\sigma_q-a^+(T),r)$\end{tiny}}
 \put(180,55){\circle{10}}     
 \put(200,45){\line(-2,1){15.5}}
 \put(200,40){\begin{tiny}$R(\sigma_q-a^-(T),r)$\end{tiny}}
 \put(100,23){\circle{10}}     
 \put(80,13){\line(2,1){15.5}}

 \put(0,10){\begin{tiny}$R(\sigma_q-\lambda_q,r)=\log(R_0)$\end{tiny}}
\put(147.5,0){$\bullet$} 
\put(152.5,0){\begin{tiny}$\log(\omega_q)$\end{tiny}}
\qbezier[100](0,2.5)(150,2.5)(300,2.5)
\put(50,-2){\begin{tiny}$\downarrow$small
radius$\downarrow$\end{tiny}}
\end{picture}\\\emph{ }\\
\end{center}

--- \emph{Step 3 : We have $R\geq \omega_q$.}\\

Indeed, if $R^-=R^+<\omega_q$, then, by the small radius Lemma \ref{q-Young},
$|a^-(T)-1|_1>|q-1|$ and $|a^+(T)-1|_1>|q-1|$. 
We shall now show
that this implies that $|a(T)-1|_1>|q-1|$, which is in
contradiction with the small radius lemma, since the 
equation $\sigma_q-a(T)$ is solvable. 
\begin{lemma}\label{|x+y+xy|=sup(|x|,|y|)}
Let $(R,|.|)$ be an ultrametric valued ring. Let $h^-,h^+\in R$ be
two elements satisfying $|h^-|<1$, and
$|h^-+h^+|=\sup(|h^-|,|h^+|)$. Then
\begin{equation}
|h^-+h^++h^-h^+|\;=\;\sup(|h^-|,|h^+|)\;.
\end{equation}
\end{lemma}
\begin{proof} If $|h^+|>|h^-|$, then 
$|h^-+h^++h^-h^+|=|h^+|$.
If $|h^+|\leq|h^-|<1$, then 
$|h^-h^+|<|h^-|=\max(|h^-|,|h^+|)=|h^-+h^+|$. 
\end{proof}

\emph{Proof of Step $3$: } Write $a^-(T)=1+h_q^-(T)$ and
$\lambda_q\cdot a^+(T)=1+(\lambda_q-1)+\lambda_q\cdot h_q^+(T)$.
Namely, in the notations of Theorem 
\ref{Thm : Motzkin annulus}, we have
$h_q^-(T)=\sum_{i\leq -1}\alpha_iT^i$ and $h^+_q(T)=\sum_{i\geq
1}\alpha_iT^i$. We apply Lemma \ref{|x+y+xy|=sup(|x|,|y|)} to the
field $R:=\mathcal{E}_K$, $h^-:=h^-_q(T)$ and
$h^+:=(\lambda_q-1)+\lambda_qh^+_q(T)$. Indeed
$|h^-+h^+|_1=\sup(|h^-|_1,|h^+|_1)$, and $|h^-|_1<1$ by Lemma \ref{|a_-i|<1}. Lemma
\ref{|x+y+xy|=sup(|x|,|y|)} then implies
\begin{equation}
|a(T)-1|_1\;=\; |(1+h^-)(1+h^+)-1|_1\;=\; 
|h^-+h^++h^-h^+|_1
\;=\;
\sup(|h^-|_1,|h^+|_1).
\end{equation}
Now, if $R^-< \omega_q$, then $|a^-(T)-1| > |q-1|$, that is
$|h^-(T)|>|q-1|$. Hence $|a(T)-1|_1>|q-1|$, which implies that the
radius of $\sigma_q-a(T)$ is small (cf. Lemma \ref{q-Young}).
Since, by assumption, $Ray(\sigma_q-a(T),1)=1$, this is absurd
and then $R\geq \omega_q$.\\

--- \emph{Step 4 : We have $R>\omega_q$.}

Since $R=R^-$ 
it is enough to show that $R^->\omega_q$.
By Lemma \ref{q-|a_i|<1} below we have 
$|a^--1|<|q-1|$. On the other hand Lemma 
\ref{q-Katz} proves that this implies
$R^->\omega_q$. 

\begin{lemma}[\protect{$q$-analogue of Lemma 
\ref{|a_i|<1}}]\label{q-|a_i|<1}
Assume that the Motzkin decomposition of 
$a(T)\in\mathcal{E}_K$ is 
$a(T):=\lambda_q a^-(T)a^+(T)$, with 
$|\lambda_q-1|<1$. If
$Ray(\sigma_q-a(T),1)>\omega_q$, then we have $a^-(T)=1+h^-_q(T)$, where $h^-_q(T)=\sum_{i\leq -1}\alpha_iT^i$ satisfies $|h_q^-|_1<|q-1|$.
\end{lemma}
\begin{proof} Consider the operator $d_q-g_{[1]}(T)$, with
$g_{[1]}(T):=\frac{a(T)-1}{(q-1)T}$, and write
\begin{equation}\label{g_1 of tensor product}
g_{[1]}(T)=\frac{a^-(T)-1}{(q-1)T}+a^{-}(T)\frac{\lambda_qa^+(T)-1}{(q-1)T}
=g_{[1]}^-(T)+a^{-}(T)\frac{\lambda_qa^+(T)-1}{(q-1)T}\;,
\end{equation}
with $g_{[1]}^-(T):=\frac{a^-(T)-1}{(q-1)T}$. Since
$Ray(d_q-g_{[1]}(T),1)>\omega_q$, hence, by 
\eqref{q-radius explicit formula} 
and Lemma \ref{limit of q-factorial}, one has
$\lim_{s\to\infty}|g_{[s]}(T)|_1=0$. In particular
$|g_{[s]}(T)|_1<1$, for some $s\geq 1$. Moreover, by the Small
Radius Lemma \ref{q-Young}, we have $|g_{[1]}(T)|_1\leq 1$. These
facts imply our claim in the following way. 

By contrapositive, suppose that
$|a^-(T)-1|_1\geq |q-1|$. Our assumption 
$Ray(\sigma_q-a(T),1)>\omega_q$ is enough to obtain 
Steps 1,2,3. In particular Step 3 says $Ray(\sigma_q-a^-(T),1)\geq \omega_q$. Then, by
Lemma \ref{q-Young}, $|g^-_{[1]}(T)|_1\leq 1$, and hence
$|a^-(T)-1|_1=|q-1|$. This means $|g^-_{[1]}(T)|_1 =1$. 

We now look to $g_{[1]}$ and get a contradiction
exploiting 
\eqref{g_1 of tensor product} and the fact that 
$|g_{[1]}^-(T)|_1= 1$. Namely, write as usual 
$a^-(T)=1+\sum_{i\geq 1}\alpha_{-i}T^{-i}$.
Let $-d\leq -1$ be the smallest index such that
$|\alpha_{-d}|=|q-1|$. Observing equation 
\eqref{g_1 of tensor product} we see that by
Lemma \ref{|a_-i|<1} the reduction $a^-(T)$ is $1$, 
and the reduction of $\frac{\lambda_qa^+(T)-1}{(q-1)T}$ lies in $t^{-1}k[[t]]$, so the 
reduction of $g_{[1]}(T)$ in $k(\!(t)\!)$ 
is of the form 
$\overline{g_{[1]}(T)}= 
\overline{\alpha}t^{-d-1}+(\textrm{terms of higher degree})$, 
where $\overline{\alpha}$ is the reduction of 
$\alpha_{-d}/(q-1)$. 

A simple 
induction on the equation
$g_{[s+1]}=d_q(g_{[s]})+\sigma_q(g_{[s]})g_{[1]}$ shows that
$\overline{g_{[s]}(T)} = \overline{\alpha}^st^{(-d-1)s}+(\textrm{terms of higher degree})$,
this is in contradiction with the fact that $|g_{[s]}(T)|_1<1$,
for some $s\geq 1$.\end{proof}

\begin{lemma}[\protect{$q$-analogue of Lemma 
\ref{Katz}}]\label{q-Katz}
Let $q\in\mathrm{D}^-(1,1)$. Let $\dq-g(T)$,
$g(T)\in\mathcal{E}_K$, be some equation. Suppose that
$|g(T)|_1\leq 1$. Then $Ray(\dq-g(T),1)>\omega_q$ if and only if
$|g_{[s]}(T)|<1$, for some $s\geq 1$.
\end{lemma}
\begin{proof} 
Condition $|g(T)|_1\leq 1$ guarantee 
that 
$n\mapsto|g_{[n]}|_1$ is decreasing. Indeed,
$|g_{[1]}|_1=|T^{-1}g(T)|_1\leq 1$ and inductively
$|g_{[n+1]}|_1=
|d_q(g_{[n]})+\sigma_q(g_{[n]})g_{[1]}|_1\leq
\sup(|g_{[n]}|_1,
|g_{[n]}g_{[1]}|_1)=|g_{[n]}|_1\sup(1,|g_{[1]}|_1)=
|g_{[n]}|_1$. So if $Ray(d_q-g(T),1)>\omega_q$, it 
follows from \eqref{q-radius explicit formula} that 
$\lim_n|g_{[n]}(T)|_1=0$.

Assume now that $|g_{[n]}|_1<1$, for some $n\geq 1$. 
Since the sequence $n\mapsto|g_{[n]}|_1$ 
is decreasing, there exists $h>0$ 
such that $|g_{[p^h]}|_1<1$. We now fix such an $h$, 
and we obtain an estimation of $Ray(d_q-g(T),1)$.
By \cite[1.2.2]{DV-Dwork}, one has
$$d_q^{(m+1)p^h}(y)=d_q^{p^h}(d_q^{mp^h}(y))=
\sum_{r=0}^{p^h}\binom{p^h}{r}_q d_q^{r}(g_{[mp^h]})\cdot
\sigma_q^r(g_{[p^h-r]})\sigma_q^r(y)\;.$$ Then
$g_{[(m+1)p^h]}=\sum_{r=0}^{p^h}\binom{p^h}{r}_q
d_q^{r}(g_{[mp^h]})\cdot \sigma_q^r(g_{[p^h-r]})a(T)a(qT)\cdots
a(q^{r-1}T)$. Now for all $j\geq 0$ one has 
$|a(q^jT)|_1=|a(T)|_1=1$, and on the other hand 
$|d_q^k(f)|_1\leq |[k]^!_q||f|_1$ 
(cf. Lemma \ref{d_q(f) leq k! f}).
Moreover $|\tbinom{p^h}{r}_q|=|[p^h]_q[p^h-1]_q\cdots
[p^h-r+1]_q|/|[r]^!_q|$, where $[n]_q:=\frac{q^n-1}{q-1}$. Since
$|[p^h]_q|<|[p]_q|$, we obtain
\begin{equation}\label{..E..}
|g_{[(m+1)p^h]}|_1\;\leq\; 
\sup(|[p]_q|,|g_{[p^h]}|_1)\cdot
|g_{[mp^h]}|_1\;.
\end{equation}
We deduce that for all $m\geq 1$ one has 
$|g_{[mp^h]}|_1\leq s^m$, where 
$s:=\sup(|[p]_q|,|g_{[p^h]}|_1)<1$.  

Now we obtain a similar estimation for all $n\geq p^h$. 
We let $m(n):=[n/p^h]\geq 1$, where $[a]$ is the
greatest integer smaller than or equal to $a$. Then 
$m(n)p^h\leq n$ and $|g_{[n]}|_1\leq |g_{[m(n)p^h]}|_1\leq s^{m(n)}$. 

Finally we now obtain the required estimation. We have
\begin{equation}\label{yyyyo}
\left|\frac{g_{[n]}}{[n]_q^!}\right|_1^{\frac{1}{n}} 
\;\leq\;
\frac{s^{m(n)/n}}{|[n]_q^!|^{1/n}}
\;\leq\;
\frac{s^{1/p^h}}{
|[n]_q^!|^{1/n}}
\;\xrightarrow[]{\;\;n\;\mapsto\infty\;\;}\;
\frac{s^{1/p^h}}{\omega_q}\;.
\end{equation}
%
By \eqref{q-radius explicit formula}, this gives
$Ray(\dq-g_{[1]},1)\geq \omega_q/s^{1/p^h}>\omega_q$. \end{proof}

\begin{hypothesis}\label{hypothesis q-1<omega}
From now on we will suppose that
$|q-1|<\omega$. This implies $\omega_q=\omega$.
\end{hypothesis}

Hypothesis \ref{hypothesis q-1<omega} 
is necessary to have Theorem 
\cite{DV-Dwork}: the antecedent by Frobenius.\\

--- \emph{Step 5: Since $|q-1|<1$, and since $R>\omega$, then,
by \cite{DV-Dwork}, we can take the antecedent by Frobenius of
$\sigma_q-a^-(T)$, $\sigma_q-a^+(T)$ and
$\sigma_q-\lambda_q$.}\\

More precisely, there exist a finite extension $K^{(1)}/K$, an
$f^+(T)=\sum_{i\geq 0}b^+_iT^i\in\a_{K^{(1)}}([0,1[)^\times$,
$f^-(T)=\sum_{i\leq 0}b^-_iT^i\in\a_{K^{(1)}}([1,\infty])^\times$,
and there are functions $a^{(1),-}(T)=\sum_{i\leq
0}\alpha_i^{(1),-}T^i\in\mathcal{E}_{K^{(1)}}$,
$a^{(1),+}(T)=\sum_{i\geq
0}\alpha_i^{(1),+}T^i\in\mathcal{E}_{K^{(1)}}$, and
$\lambda_q^{(1)}\in K^{(1)}$ such that
\begin{eqnarray*}
(\lambda_q^{(1)})^p & = & \lambda_q\;;\\
a^{(1),-}(T^p)^\sigma \cdot a^{(1),-}(q
\cdot T^{p})^\sigma\cdots
a^{(1),-}(q^{p-1}T^{p})^\sigma&=&
a^-(T)\cdot\frac{f^-(q\cdot T)}{f^-(T)}\;;\\
a^{(1),+}(T^p)^\sigma 
\cdot a^{(1),+}(q\cdot T^{p})^\sigma\cdots
a^{(1),+}(q^{p-1}T^{p})^\sigma&=& 
a^+(T)\cdot
\frac{f^+(q\cdot T)}{f^+(T)}\;,
\end{eqnarray*}
where, for all functions 
$a(T):=\sum \alpha_iT^i\in\mathcal{E}_K$,
we let $a(T)^\sigma:=\sum\sigma(\alpha_i)T^{i}$. 

These conditions imply 
immediately that $b_0^+\neq 0$, $b_0^-\neq 0$, and that
$f^+(qT)/f^+(T)) = 1+u_1T+u_2T^2+\cdots$, and $f^-(qT)/f^-(T) =
1+u_{-1}T^{-1}+u_{-2}T^{-2}+\cdots$. Since
$a^-(T)=1+\alpha_{-1}T^{-1}+\cdots$, and
$a^+(T)=1+\alpha_1T+\cdots$, this implies that
$\alpha_0^{(1),+}=1$ and $\alpha_0^{(1),-}=1$. Hence the function
\begin{equation}
a^{(1)}(T)\;:=\;
\lambda_q^{(1)}\cdot a^{(1),-}(T)\cdot a^{(1),+}(T)
\end{equation}
lies in $\mathcal{E}_K$, and it is the Motzkin 
decomposition of Theorem \ref{motzkin}. 
Observe now that both $f^-$ and
$f^+$ belong to $\mathcal{E}^{\times}_K$, hence
$\sigma_q-a^{(1)}(T)$ is an antecedent of Frobenius of
$\sigma_q-a(T)$ over $\mathcal{E}_K$, and it is then 
solvable.

--- \emph{Step 6 : } Steps $1$, $2$, $3$, $4$ are still true for the antecedent. In
particular if
\begin{eqnarray}
R^-(1)&:=&Ray(\d-g^{(1),-}(T),1)\;,\\
R^+(1)&:=&Ray(\d-g^{(1),+}(T),1)\;,\\
R^0(1)&:=&Ray(\d-b_0,1)\;.
\end{eqnarray}
we must have $R^-(1)=R^{+}(1)>\omega$. Let
$R(1):=R^{-}(1)=R^{+}(1)$, then $R(1)=R^{1/p}$, by the property of
the antecedent. This implies $R>\omega^{1/p}$. 

Now the process can be iterated since 
$R(1)>\omega$, and we can again consider 
the antecedent. This shows that $R>\omega^{1/p^h}$,
for all $h\geq 0$, that is $R=1$. 
Proposition \ref{q-division of
the problem over Amice} hence follows. \end{proof} 

\begin{corollary}[q-analogue of \ref{g^+ is trivial}]\label{q-analogue of g^+ is trivial}
Let $q\in\mathrm{D}^-(1,1)$. Let $\sigma_q-a(T)$ be a solvable
differential equation. Let $a(T)=\lambda\cdot a^{-}(T)\cdot
a^{+}(T)$ be the Motzkin decomposition of $a(T)$. Then
$\lambda=q^{a_0}$, for some $a_0\in K$. Moreover, this operator is
isomorphic to $\sigma_q-\lambda\cdot a^{-}(T)$.
\end{corollary}
\begin{proof} See the proof of \ref{g^+ is trivial}. \end{proof}

\begin{remark}
The unique obstruction to generalize Proposition 
\ref{q-division of the problem over Amice} and Corollary 
\ref{q-analogue of g^+ is trivial}
to the case $|q-1|<1$ is represented by the so called
Weak Frobenius structure for $q-$difference modules 
over a disk with $|q-1|<1$. This is proved in 
\cite{DV-Dwork} in the case $|q-1|<\omega$. 

The assumption $|q-1|<\omega$ is also used in the 
sequel, where we consider logarithms of the 
exponentials. E.g. see Step 0 of Lemma 
\ref{q-criteria of solvability lemma}.
\end{remark}

\subsection{Criterion of Solvability}
\begin{lemma}[q-analogue of 
\ref{criteria of solvability lemma2}]
\label{q-criteria of solvability lemma}
Let $|q-1|<\omega$. Suppose that $a(T)=a^{+}(T)$ is
the Motzkin decomposition of $a(T)$. Write
$a^+(T)=\exp(\sum_{i\geq 1}a_iT^i)\in\a([0,1[)^{\times}$ (cf. the
settings of \ref{the settings}). Then the $q-$difference equation
$\sigma_q-a^+(T)$ is solvable if and only if there exists a family
$\{\lb_{n}\}_{n\in\J}$, where $\lb_n\in\W(\O_K)$ has phantom
components $\phi_{n}=(\phi_{n,0},\phi_{n,1},\ldots)$ 
satisfying
\begin{equation}\label{q-a_np^m=n phi_n,m}
a_{np^m}=\frac{(q^{np^m}-1)}{p^m}\cdot 
\phi_{n,m}\;,\qquad\textrm{
for all }n\in\J,\; m\geq 0\;,
\end{equation}
for all $n\in\J$, and all $m\geq 0$. In other words, the 
formal solution of the equation $\sigma_q(y)=ay$ is 
\eqref{formal solution explicit}
\begin{equation}
y(T)\;=\;
E(\sum_{n\in\J}\lb_nT^n,1)\;:=\;
\exp(\sum_{n\in\J}\sum_{m\geq
0}\phi_{n,m}\frac{T^{np^m}}{p^m})\;.
\end{equation}
\end{lemma}
\begin{proof}
The formal
series $E(\sum_{n\in\J}\lb_nT^n,1)$ belongs to
$1+T\cdot\mathfrak{p}_K[[T]]\subset\mathcal{E}_K$, and it is solution
of the equation $L:=\sigma_q-\exp(\sum_{n\in\J}\sum_{m\geq
0}\phi_{n,m}(q^{np^m}-1)T^{np^m})$. 
Since this 
exponential
converges in the open unit disk, then 
$Ray(L,\rho)=\rho$, for all
$\rho<1$. Hence, by continuity of the radius, 
$Ray(L,1)=1$ and $L$
is solvable. 

Conversely, suppose that $\sigma_q-a^+(T)$ is
solvable, then the Witt vectors
$\lb_n=(\lambda_{n,0},\lambda_{n,1},\ldots)$ are 
defined by the relation \eqref{q-a_np^m=n phi_n,m}. For example, for all $n\in\J$
we have
\begin{equation}
\lambda_{n,0} = \frac{a_{n}}{(q^{n}-1)} \quad,\qquad \lambda_{n,1}
= \frac{1}{p}\left(\frac{p\cdot a_{np}}{(q^{np}-1)} -
\Bigl(\frac{a_n}{(q^{n}-1)}\Bigr)^{\!p} \right)\;.
\end{equation}
We must show that $|\lambda_{n,m}|\leq 1$, for all 
$n\in\J$, $m\geq 0$.\\

--- \emph{Step $0$ : We have $|\lambda_{n,0}|=|\phi_{n,0}|\leq 1$ for all
$n\in\J$.}\\

This results by the small radius Lemma \ref{q-Young} as 
follows: denote the argument of the
exponential $a^+(T)$ by 
$\phi^+_q(T):=\sum_{n\in\J}\sum_{m\geq
0}\phi_{n,m}(q^{np^m}-1)T^{p^m}/p^m$. 
By Lemma \ref{q-Young}, one has
$|a^+(T)-1|_1=|\exp(\phi_q^+)-1|_1\leq |q-1|$. Since
$|q-1|<\omega$, then $|\exp(\phi_q^+)-1|_1<\omega$, hence
$\phi_q^+=\log(\exp(\phi_{q}^+))$ and
$|\phi_q^+|_1=|\exp(\phi_q^+)-1|_1\leq|q-1|$. This implies
$|\phi_{n,m}(q^{np^m}-1)/p^m|\leq |q-1|$, for all $n\in\J$ and all
$m\geq 0$. In particular, for $m=0$ we have
$|\lambda_{n,0}|=|\phi_{n,0}|\leq 1$, for all $n\in\J$.\\

--- \emph{Step $1$ :} By Step 0 the exponential
$$E(\sum_{n\in\J}(\lambda_{n,0},0,0,\ldots)T^{n},1)=
\exp(\sum_{n\in\J}\sum_{m\geq
0}\lambda_{n,0}^{p^m}\frac{T^{np^m}}{p^m})$$ converges in the unit
disk and is solution of the operator $Q^{(0)}:=\sigma_q-
a^{(0)}(T)$, with $a^{(0)}(T)=\exp(\sum_{n\in\J}\sum_{m\geq
0}\lambda_{n,0}^{p^m}(q^{np^m}-1)\frac{T^{np^m}}{p^m})$.
$Q^{(0)}$ is then solvable.\\

--- \emph{Step 2 :}  The tensor product operator $\sigma_q - (a^+(T)/a^{(0)}(T))$
is again solvable. We have explicitly
$$\frac{a^+(T)}{a^{(0)}(T)}=
\exp(\sum_{n\in\J}\sum_{m\geq
0}(\phi_{n,m}-\lambda_{n,0}^{p^m})(q^{np^m}-1)\frac{T^{np^m}}{p^m})\;.$$
This operator corresponds to the family of  Witt vectors
$\{\lb_n-(\lambda_{n,0},0,0,\ldots)=(0,\lambda_{n,1},\lambda_{n,2},\ldots)\}_{n\in\J}$.
Observe that the coefficient corresponding to $m=0$ is equal to
$0$, for all $n\in\J$. This leads us to compute easily the
``antecedent by ramification'' of $\sigma_q - a^+(T)/a^{(0)}(T)$,
namely this antecedent is given by $\sigma_q-a^{(1)}(T)$, with
$$a^{(1)}(T):=\exp(\sum_{n\in\J}\sum_{m\geq 0}
(\phi_{n,m}-\lambda_{n,0}^{p^m})(q^{np^m}-1)\frac{(q-1)}{(q^{p}-1)}\frac{T^{np^{m-1}}}{p^m})\;.$$
In other words, we have
$$a^{(1)}(T^p)\cdot a^{(1)}(qT^p)\cdot a^{(1)}(q^2T^p)\cdots
a^{(1)}(q^{p-1}T^p)=\frac{a^+(T)}{a^{(0)}(T)}\;.$$

--- \emph{Step 3 :} The antecedent
is again solvable, hence, as in Step $0$, we find
$|\phi_{n,1}-\lambda_{n,0}^p|\leq |q^p-1|=|p|$, which implies
$|\lambda_{n,1}|\leq 1$. 

The process can be iterated indefinitely.
\end{proof}

\if{
\begin{remark}\label{what happens when q-1>1 ?}
For $|q-1|\geq \omega$ we do not know whether there 
exist solvable rank one
$q-$difference equations that are not strongly confluent 
(cf. terminology of \cite{Pu-q-Diff}). Such an equation is 
such that its solution $y(T)$ is an exponential of type
$E(\sum_{n\in\J}\lb_nT^n,1)$, which lies in 
$\O_K[[T]]$ but such
that $\lb_n\in \W(K)-\W(\O_K)$, for some $n\in\J$. 
The author does
not know examples of such Witt vectors.
\end{remark}
}\fi

\begin{remark}[q-analogue of Remark 
\ref{discussion}]\label{q-discussion} 
We shall now consider the general
case of an equation $\sigma_q-a(T)$, with 
$a(T)=\lambda\cdot
a^{-}(T)a^{+}(T)\in\mathcal{E}_K$, and get a criteria 
of solvability. 
We proceed as in Remark \ref{discussion}. Suppose given two
families $\{\lb_{-n}\}_{n\in\J}$ and $\{\lb_{n}\}_{n\in\J}$, with
$\lb_n\in\W(\O_K)$. By Lemma 
\ref{a^+ belong always to E} below, $a^+(T)$
belongs always to $\mathcal{E}_K$. On the other hand, 
we will prove 
(cf. Lemma \ref{q-criteria for belong to Amice}) that the 
series
$a^-(T)$ belongs to $\mathcal{E}_K$ if and only if the 
family
$\{\lb_{-n}\}_{n\in\J}$ belongs to 
$\mathrm{Conv}(\mathcal{E}_K)$ 
(cf. Definition \ref{Conv}).
\end{remark}

\begin{notation}\label{q-notations}
Let $\sigma_q- a(q,T)$, $a(q,T)\in\mathcal{E}_K$ be a solvable
differential equation. Let $a(q,T):=q^{a_0}\cdot a^-(q,T)\cdot
a^+(q,T)$, $a_0\in\mathbb{Z}_p$, be the Motzkin decomposition of
$a(q,T)$. In the notations of Lemma \ref{q-criteria of solvability
lemma} we can write
\begin{equation} a^-(q,T)=\exp(\phi_q^-(T))\quad,\quad
a^+(q,T)=\exp(\phi_q^+(T))\;,
\end{equation}
\begin{eqnarray}\label{gygy}
\phi_q^-(T)&:=&\sum_{n\in\J}\sum_{m\geq
0}\phi_{-n,m}(q^{-np^m}-1)\frac{T^{-np^m}}{p^m}\;,\\
\label{gygyg}\phi_q^+(T)&:=&\sum_{n\in\J}\sum_{m\geq
0}\phi_{n,m}(q^{np^m}-1)\frac{T^{np^m}}{p^m}\;.
\end{eqnarray}
For all $n\in\J$ we denote by 
$\bs{\lambda}_n,\bs{\lambda}_{-n}\in\W(K)$ the 
Witt vectors with phantom vectors 
$\lr{\phi_{n,0},\phi_{n,1},\ldots}$ and
$\lr{\phi_{-n,0},\phi_{-n,1},\ldots}$ respectively.
In other words, the solution of $\sigma_q-a(q,T)$ can be
represented by the symbol
\begin{equation}\label{q-formal solution}
y(T):=T^{a_0}\cdot\exp(\sum_{n\in\J}\sum_{m\geq
0}\phi_{-n,m}\frac{T^{-np^m}}{p^m})\cdot\exp(\sum_{n\in\J}\sum_{m\geq
0}\phi_{n,m}\frac{T^{np^m}}{p^m})\;,
\end{equation}
as well as for differential equations.
\end{notation}
\begin{lemma}\label{a^+ belong always to E}
Let $|q-1|<\omega$. Let $\{\lb_n\}_{n\in\J}$ be a family of Witt
vectors such that $\lb_n\in\W(\O_K)$. Then $a^+(T)$ belongs to
$\mathcal{E}_K$.
\end{lemma}
\begin{proof} We use the notations of \cite{Rk1}. 
Let
$P(X)=(X+1)^p-1$ be the Lubin-Tate series 
corresponding to the
formal multiplicative group $\hat{\mathbb{G}}_m$. The 
phantom vector of $[q^n-1]_{P}\in\W(\O_K)$ is then
$\ph{q^n-1,q^{np}-1,q^{np^2}-1,\cdots}$, for all 
$n\in\mathbb{Z}$. Then, for all $n\in\J$, the phantom
vector of $[q^n-1]_{P}\cdot \lb_{n}$  is
$$\ph{(q^{n}-1)\phi_{n,0}\;,\;(q^{np}-1)\phi_{n,1}\;,\;(q^{np^2}-1)\phi_{n,2}\;,\;\ldots}\;.$$
Hence we can express $a^+(q,T)$ as a product of Artin-Hasse
exponentials
$$a^+(q,T)=\prod_{n\in\J}E([q^{n}-1]_{P}\cdot \lb_{n}\;,\;T)\;.$$
Since $[q^{n}-1]_{P}\cdot\lb_n\in\W(\O_K)$, then 
for all $n\in\J$ the Artin-Hasse exponential $E([q^{n}-1]_{P}\cdot
\lb_{n}\;,\;T)$ belongs to $1+T\O_K[[T]]$, which is contained in
$\mathcal{E}_K$. 
\end{proof}
\begin{lemma}[q-analogue of Proposition 
\ref{criteria for belong to Amice}]\label{q-criteria for belong to Amice}\label{a(T) belong to E_k}
Let $|q-1|<\omega$. Let $\{\lb_{-n}\}_{n\in\J}\in\W(\O_K)$ be a
family of Witt vectors. Then the
following assertions are equivalent:\\
$(1)$ The series $a^-(T)=\exp(\phi_q^-(T))$ belongs to $\mathcal{E}_K$;\\
$(2)$ $\phi_q^-(T)$ belongs to $\mathcal{E}_K$;\\
$(3)$ $\{\lb_{-n}\}_{n\in\J}\in\mathrm{Conv}(\mathcal{E})$ (cf. Definition \ref{Conv}).
\end{lemma}
\begin{proof} 
The equivalence $(2)\Leftrightarrow (3)$ follows from 
Proposition \ref{criteria for belong to Amice}. 

We firstly observe that since by assumption we have
$\lb_{-n}\in\W(\O_K)$, then $|\phi_{-n,m}|\leq 1$ and 
so
\begin{equation}\label{phi_q^-<omega}
|(q^{-np^m}-1)\phi_{-n,m}p^{-m}| = 
|q-1|\cdot|\phi_{-n,m}|
\leq|q-1|< \omega\;.
\end{equation}
Hence $|\phi_q^-(T)|_1\leq |q-1|<\omega$. 

Now assume that
$\phi_q^-(T)\in\mathcal{E}_K$. Since the exponential 
series converges in the disk 
$D_{\mathcal{E}_K}^-(0,\omega):= 
\{f\in\mathcal{E}_K\;|\;|f|_1<\omega\}$,
then $\exp(\phi_q^-(T))\in\mathcal{E}_K$. 

Conversely, assume that
$\exp(\phi_q^-(T))\in\mathcal{E}_K$. Since, for all $\rho>1$,
$|\phi_q^-(T)|_\rho<|q-1|$, then
$\phi_q^-(T)\in\mathrm{D}_{\a_K([\rho,\infty])}^{-}(0,\omega):=\{f\in\a_K([\rho,\infty])\;|\;|f|_\rho<\omega\}$,
and hence  $\exp(\phi_q^-(T))$ converge in $\a_K([\rho,\infty])$,
for all $\rho>1$. Moreover,
$|\exp(\phi_q^-(T))-1|_\rho=|\phi_q^-(T)|_\rho\leq|q-1|<\omega$,
for all $\rho>1$. By continuity, we have
$|\exp(\phi_q^-(T))|_1=|\phi_q^-(T)|_1 \leq |q-1|<\omega$. Now the
logarithm converges in the disk
$D_{\mathcal{E}_K}(1,1^-):=\{f\in\mathcal{E}_K\;|\;|f|_1<1\}$,
hence $\phi_q^-(T)=\log\exp(\phi_q^-(T))$. Then $\phi_q^-(T)$
belongs to $\mathcal{E}_K$.
\end{proof}

\begin{corollary}[Criterion of solvability for $q$-difference
equations]\label{criterion of solvability for q-difference} The
equation $\sigma_q-a(q,T)$, with $a(q,T)=\lambda_q T^N a^-(q,T)
a^+(q,T)$, with
$$a^-(T):=\exp(\sum_{i\leq -1}a_iT^i)\quad,\quad a^+(T):=\exp(\sum_{i\geq
1}a_iT^i)\;,$$ is solvable if and only if the following conditions
are verified
\begin{enumerate}
 \item $\lambda=q^{a_0}$, with $a_0\in\mathbb{Z}_p$\;;
 \item $N=0$\;;
 \item There exist two families $\{\lb_{-n}\}_{n\in\J}$ and
 $\{\lb_{n}\}_{n\in\J}$, with $\lb_{-n},\lb_{n}\in\W(\O_K)$, for
 all $n\in\J$, such that
\begin{equation}
a_{-np^m}=\frac{(q^{-np^m}-1)}{p^m}\cdot\phi_{-n,m}\quad,\quad
a_{np^m}=\frac{(q^{np^m}-1)}{p^m}\cdot\phi_{n,m}\;,
\end{equation}
 for all $n\in\J$ and all $m\geq 0$;
 \item $\{\lb_{-n}\}_{n\in\J}\in\mathrm{Conv}(\mathcal{E})$.
\end{enumerate}
In other words, the formal solution of this equation can be
represented by the symbol \eqref{q-formal solution} in which the
family $\{\lb_{-n}\}_{n\in\J}$ belongs to
$\mathrm{Conv}(\mathcal{E}_K)$, and $a(T)=\exp(\phi^-_q(T))\cdot
q^{a_0}\cdot\exp(\phi^+_q(T))$, where $\phi^-_q(T)$, $\phi^+_q(T)$
are defined in \eqref{gygy} and \eqref{gygyg}.\hfill$\Box$
\end{corollary}

\begin{corollary}[canonical extension for $q-$difference]
Let $\sigma_q-Mod(\a_K([1,\infty]))^{\mathrm{sol}}_{rk=1}$ be the
category of rank one $\sigma_q-$modules over $\a_K([1,\infty])$,
solvable at all $\rho\geq 1$. The scalar extension functor
$\sigma_q\textrm{-}\mathrm{Mod}(\a_K([1,\infty]))^{\mathrm{sol}}_{rk=1}
\to
\sigma_q\textrm{-}\mathrm{Mod}(\mathcal{E}_K)^{\mathrm{sol}}_{rk=1}$
is an equivalence.
\end{corollary}
\begin{proof} 
The proof is analogous 
to the proof of Corollary 
\ref{canonical ext over
amice}. \end{proof}

\begin{remark}[Strong confluence]
\label{strongly confluence}

The $q$-deformation and $q$-confluence equivalences 
of \cite{Pu-q-Diff} do not hold  over the ring 
$\mathcal{E}_K$. Indeed those equivalences involve
the Taylor solutions, and their convergence locus. 
The Taylor solution of a differential equations over 
$\mathcal{E}_K$ does not converge anywhere.

However the computations we have obtained show that 
the solutions of differential equations and of 
$q-$difference equation over $\mathcal{E}_K$ coincide.
Moreover by the canonical extension theorem for 
differential and $q-$difference equations one knows that, 
if $|q-1|<\omega$, then
every rank one object comes by scalar extension from an 
object over the affinoid domain
$A:=\mathbb{P}^1-\mathrm{D}^-(0,1)=
\{|x|\geq 1\}$. 
In particular, for all $r>1$, every object comes by 
scalar extension from an object over the 
closed annulus $\{|x|\in[1,r]\}$. 
Hence we can apply the deformation and the confluence 
to the canonical extensions. 
\end{remark}

\bibliographystyle{amsalpha}
\bibliography{2012-NP-III}
\end{document}

%% file: raccourcisandrea.tex
\renewcommand{\a}{\mathcal{A}}

\newcommand{\bs}[1]{\boldsymbol{#1}}

\newcommand{\CW}{\bs{\mathrm{CW}}}

\newcommand{\dem}{\begin{proof}}

\renewcommand{\d}{\partial_T}
\newcommand{\dq}{\Delta_q}

\newcommand{\Ed}{\mathcal{E}^{\dag}}

\newcommand{\Fb}{\overline{\mathrm{F}}}

\newcommand{\Hom}{\mathrm{Hom}}

\newcommand{\J}{\mathbb{J}}
\newcommand{\M}{\mathrm{M}}

\newcommand{\lr}[1]{\langle#1\rangle}
\newcommand{\lb}{\bs{\lambda}}

\renewcommand{\O}{\mathscr{O}}

\newcommand{\ph}[1]{\langle#1\rangle}

\newcommand{\simto}{\stackrel{\sim}{\to}}

\newcommand{\W}{\bs{\mathrm{W}}}